\documentclass[12pt]{article}
\usepackage[margin=1in]{geometry} 
\usepackage{amsmath,amsthm,amssymb}
\usepackage{comment}
\usepackage{graphicx}
\usepackage{subcaption}
\usepackage[font=small]{caption}

\usepackage{xcolor}
\usepackage{hyperref}
\hypersetup{
    colorlinks,
    linkcolor={blue!90!black},
    citecolor={blue!100!black},
    urlcolor={blue!80!black}
}

\usepackage{tikz}
\usetikzlibrary{decorations.pathreplacing,calligraphy}
\usetikzlibrary{arrows} 

\usepackage{authblk}

\title{Asymptotic Distribution of Low-Dimensional Patterns Induced by Non-Differentiable Regularizers under General Loss Functions}

\author[1]{Ivan Hejný}
\author[1]{Jonas Wallin}
\author[2]{Małgorzata Bogdan}
\affil[1]{Department of Statistics, Lund University}
\affil[2]{Institute of Mathematics, University of Wroclaw}
\date{}

\usepackage{titlesec}
\titleformat*{\section}{\bfseries}
\titleformat*{\subsection}{\small\bfseries}

\theoremstyle{plain}
\newtheorem{theorem}{Theorem}[section]
\newtheorem{proposition}[theorem]{Proposition}
\newtheorem{lemma}[theorem]{Lemma}
\newtheorem{corollary}[theorem]{Corollary}

\theoremstyle{definition}
\newtheorem{definition}[theorem]{Definition}
\newtheorem{example}[theorem]{Example}
\newtheorem{remark}[theorem]{Remark}

\usepackage[normalem]{ulem}
\begin{document}

\maketitle

\begin{abstract}
\noindent
This article investigates the asymptotic distribution of penalized estimators with non-differentiable penalties designed to recover low-dimensional pattern structures. Patterns play a central role in estimation, as they reveal the underlying structure of the parameter — which coefficients are zero, which are equal, and how they are clustered.

The main technical challenge stems from the discontinuous nature of these patterns (such as the sign function in the case of the Lasso penalty), a difficulty not previously addressed in the literature and only recently analyzed for the standard linear model. To overcome this, we extend classical results from empirical process theory for $M$-estimation by incorporating the distributional behavior of model patterns.

We introduce a new mathematical framework for studying pattern convergence of regularized $M$-estimators. While classical approaches to distributional convergence rely on uniform conditions, our analysis employs a new local condition, \emph{stochastic Lipschitz differentiability} (SLD), which controls fluctuations of the Taylor remainder.

We demonstrate how this framework applies to a broad class of loss functions, covering generalized linear models (e.g., logistic and Poisson regression) and robust regression settings with non-smooth losses such as the Huber and quantile loss.

\end{abstract}

\section{Introduction}

Regularization plays an important role in modern statistics. A key goal is uncovering low-dimensional structures within complex datasets. This pursuit not only facilitates data compression but also enhances estimation accuracy and predictive performance by reducing the number of estimable parameters.

This paper explores how convex optimization methods can reveal such structures. We consider $M$-estimators of the form
\begin{equation}\label{main objective intro}
     \hat{\theta}_n = \underset{\theta\in\Theta}{\operatorname{argmin}} \hspace{0.1cm} \frac{1}{n}\sum_{i=1}^{n} \ell(y_i,x_i, \theta) + n^{-1/2}f(\theta),
\end{equation}
where $(y_i,x_i)$ are i.i.d. observations with $y_i \in \mathbb{R}$ and $x_i \in \mathbb{R}^p$, $\ell$ is a loss function, and $f$ a suitable regularizer. We focus on regularizers such as the Lasso, Fused Lasso, or SLOPE, each designed to identify specific low-dimensional structures \cite{tibshirani1996regression, Tib2005fusedLasso, tibshirani2011solution, bogdan2015slope, figueiredo16}: the Lasso promotes sparsity, the Fused Lasso detects contiguous groups of equal coefficients, and SLOPE detects clusters of equal coefficients (including zeros) regardless of order \cite{schneider2022geometry, Skalski2022pattern, BOGDAN_pattern, graczyk2023pattern, hejny2025unveiling}. We refer to these lower-dimensional structures as \emph{patterns}; see Section~\ref{sec:Pconv} for a precise definition. 

Regularized generalized linear and robust regression models have been extensively studied in recent years. Algorithms for regularized logistic and multinomial regression were developed in \cite{friedman2010regularization} and extended to general GLMs in \cite{tay2023elastic}. Various loss–penalty pairs were analyzed in \cite{rosset2007piecewise}, including the “Huberized” Lasso, and it was shown that all ``almost quadratic'' losses paired with $\ell_1$ regularization yield piecewise-linear solution paths. Moreover, penalized quantile regression has been investigated in \cite{wu2009variable,li2008quantile}, where model consistency and oracle properties were established for SCAD and adaptive Lasso in fixed dimensions. Further analysis of $\ell_1$-penalized quantile regression was provided in \cite{belloni2011l1}, and \cite{fan2014adaptive} proposed an adaptive robust variable selection method with a weighted $\ell_1$ penalty and model selection oracle guarantees.

In this paper we fix the dimension $p$ and derive an asymptotic framework in which the $\sqrt{n}$-rescaled estimation error admits a non-degenerate limiting distribution for a large class of loss functions.

One classical approach for proving convergence in distribution for general loss functions uses Pollard’s notion of stochastic differentiability~\cite{pollard_1985}, which guarantees that the empirical process behaves smoothly in an average sense around the true parameter. However, this condition is insufficient for establishing distributional convergence of patterns.

We introduce the concept of \emph{stochastic Lipschitz differentiability} (SLD), a sufficient condition for the desired pattern convergence of the estimator \eqref{main objective intro}.

We then show that the SLD condition holds for several important model classes, including logistic and Poisson regression, as well as robust multiple regression losses such as the Huber~\cite{Huber1964} and quantile loss~\cite{koenker1978regression}.

The main results of the paper are Theorem~\ref{main pattern robust theorem in distribution} and Theorem~\ref{main pattern robust theorem}. These extend Pollard’s Central Limit Theorem to penalized $M$-estimators, and unify several existing asymptotic theorems in the literature. In particular, Theorem~\ref{main pattern robust theorem in distribution} encompasses classical results such as Theorem~1 in~\cite{Tib2005fusedLasso} and Theorem~2 in~\cite{fu2000asymptotics} as special cases. Notably, these earlier works establish asymptotic distributions for the Lasso and Fused Lasso but do not address the crucial question of distributional pattern convergence, which was resolved only recently for the linear model in our previous work~\cite{hejny2025unveiling}. Using the SLD condition, Theorem~\ref{main pattern robust theorem} extends the pattern convergence to a broader class of loss functions.

The remainder of the paper is organized as follows:
Section 1.1 provides a brief overview of the asymptotics, while Section 1.2 sets up the model assumptions and notation. Section 2 develops a hierarchy of stochastic differentiability notions. Section 3 establishes the main result of the paper. In Theorem \ref{main pattern robust theorem in distribution} we derive the asymptotic distribution of \eqref{main objective intro}, in Theorem \ref{main pattern robust theorem} the distribution of its pattern, and finally in Corollary \ref{corollary of the main pattern theorem} we formulate the results under stronger but more easily verifiable assumptions.  Section 4 analyzes pattern recovery. Section 5 verifies the assumptions for key loss families: GLMs, Huber, and quantile regression. Section 6 presents
simulations for SLOPE-regularized logistic regression, and Section 7 concludes with a discussion.

\subsection{Overview of the asymptotics}
We distinguish between four notions of differentiability: regular, stochastic, Lipschitz stochastic, and local, forming a hierarchy illustrated in Figure~\ref{fig:diamond-graph}.

\begin{figure}[ht]
    \centering
    \begin{tikzpicture}
        % Define the nodes in a diamond shape
        \node (B) at (0, 3) {``regular'' (\ref{regularity conditions})}; % Top node
        \node (A) at (-2, 1.5) {``stochastic'' (\ref{stochastic differentiability})}; % Left node
        \node (C) at (2, 1.5) {``SLD'' (\ref{stochastic Lipschitz differentiability condition})}; % Right node
        \node (D) at (0, 0) {``local'' (\ref{local stochastic differentiability})}; % Bottom node
  
        \draw[black,-{latex}] (B) -- (A);
        \draw[black,-{latex}] (B) -- (C);
        \draw[black,-{latex}] (A) -- (D);
        \draw[black,-{latex}] (C) -- (D);

    \end{tikzpicture}
    \caption{Different notions of stochastic differentiability}
    \label{fig:diamond-graph}
\end{figure}
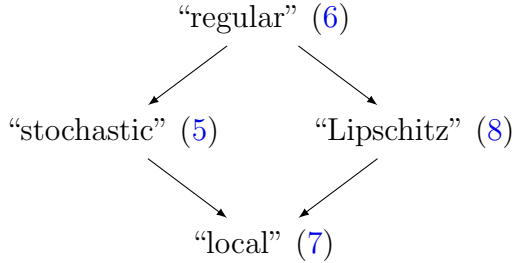

%We show that the stochastic Lipschitz differentiability condition is relatively mild and follows from standard regularity assumptions on the loss function. 
Classical conditions ensuring the asymptotic normality of the unpenalized estimator $\hat{\theta}_n$ in \eqref{main objective intro} (with $f = 0$) typically require that the loss be twice continuously differentiable. We formulate these conditions in \eqref{regularity conditions}. Under these conditions, if $\hat{\theta}_n$ is consistent for $\theta_0$ and the Hessian, $C$, of $G(\theta)$ at $\theta_0$ is non-singular, then the scaled estimation error $\sqrt{n}(\hat{\theta}_n - \theta_0)$ is asymptotically normal; see for example Theorem~2 in \cite{pollard_1985}. Further background and references can be found in Chapter~5 of \cite{van1998asymptotic} and Chapter~3.2 of \cite{vanderVaart1996}.
%with mean zero and covariance $C^{-1}C_{\triangle}C^{-1}$, where $C_{\triangle}$ is given by $\mathbb{E}[\ell'(x, \theta_0)\ell'(x, \theta_0)^T]$; see for example Theorem~2 in \cite{pollard_1985}. 

Under the same assumptions, a corresponding result holds for a broad class of penalized problems ($f\neq 0$). In this case, the scaled estimation error $\sqrt{n}(\hat{\theta}_n - \theta_0)$ converges in distribution to the minimizer of the strictly convex random function
\begin{equation*}
    V(u) = \frac{1}{2} u^T C u - u^T W + f'(\theta_0; u),
\end{equation*}
where $W$ is a zero mean multivariate Gaussian random variable and $f'(\theta_0; u)$ denotes the directional derivative of $f$ at $\theta_0$ in the direction $u$.
The classical differentiability conditions on $\ell$ are relatively concrete and straightforward to verify, yet they remain restrictive in practice: they exclude many important losses that are not twice differentiable, such as the robust Huber and quantile losses.

However, these regularity conditions are not necessary for asymptotic normality in the unpenalized case, nor for convergence in distribution of the rescaled error %to the limit above 
when a penalty $f$ is included. A considerably weaker and more general requirement is Pollard’s notion of stochastic differentiability of the loss $\ell$ %$\ell(x, \theta)$ 
at $\theta_0$, combined with twice continuous differentiability of the population risk $G(\theta)$ in a neighborhood of $\theta_0$; see Theorem 2 in \cite{pollard_1985}. Importantly, the stochastic differentiability condition accommodates several practically relevant losses that violate twice differentiability, including the robust Huber and quantile losses.

%Central to our analysis is the study of model patterns, which are induced by regularization. 
While these asymptotic results describe the distribution of $\hat{\theta}_n$, they do not automatically guarantee convergence at the level of patterns induced by $f$. A key question is whether the distributional convergence of $\sqrt{n}(\hat{\theta}_n-\theta_0)$ translates into distributional convergence of the pattern. This is not obvious and requires careful, separate treatment. Because patterns are discontinuous, the continuous mapping theorem does not apply.

For penalties of polyhedral type, $f(\theta)=\max\{v_1^T\theta,\dots,v_k^T\theta\} $,
where $v_1,\dots,v_k$ are regularizer-specific vectors in $\mathbb{R}^p$, distributional convergence of patterns was established for the linear model in \cite{hejny2025unveiling}. Stochastic Lipschitz differentiability extends the convergence of patterns induced by polyhedral regularization to general models.

\begin{subsection}{Model assumptions and notation}
We first reformulate the main objective \eqref{main objective intro} in a more general framework, indicating that our methods are applicable beyond regression, such as to graphical models. Let $\{\mathbb{P}_{\theta}: \theta\in\Theta\}$ be a family of probability measures on $\mathbb{R}^d$, indexed by the set of parameters $\Theta\subset\mathbb{R}^p$. Consider  
\begin{equation}\label{main objective}
     \hat{\theta}_n = \underset{\theta\in\Theta}{\operatorname{argmin}} \hspace{0.1cm} \frac{1}{n}\sum_{i=1}^{n} \ell(X_i, \theta) + n^{-1/2}f(\theta),
\end{equation} 
where \(X_1,\dots,X_n\) are i.i.d.\ random observations on \(\mathbb{R}^d\) sampled from a probability measure \(\mathbb{P}\), which typically equals $\mathbb{P}_{\theta}$ for some  $\theta\in\Theta$. However, we also allow for misspecified distributions $\mathbb{P}$ outside of the model.

We study the estimator \(\hat{\theta}_n\) defined in \eqref{main objective}, where 
\(f\) is an admissible convex penalty of polyhedral type (specified later in \eqref{penalty form}), 
and \(\mathcal{F} = \{\ell(x, \theta) : \theta \in \Theta\}\) is the family of loss functions under consideration.

We denote the population risk by
\begin{equation*}
    G(\theta) := \mathbb{E}[\ell(X, \theta)].
\end{equation*}

Throughout the paper we assume the following mild baseline conditions:
\begin{itemize}
    \item For every \(\theta\in\Theta\), the map \(x\mapsto\ell(x,\theta)\) is Borel-measurable.
    \item For every \(x\in\mathbb{R}^d\), the map \(\theta\mapsto\ell(x,\theta)\) is continuous on \(\Theta\).
    \item The population risk \(G(\theta)\) admits a unique minimizer \(\theta_0\), and there exists an open neighborhood \(B\subset\mathbb{R}^p\) with \(\theta_0\in B\subset\Theta\).
\end{itemize}

In what follows, we study conditions on the family \(\mathcal{F}\) under which the rescaled estimator
\[
    \hat{u}_n := \sqrt{n}\,(\hat{\theta}_n - \theta_0)
\]
converges in distribution to some random vector \(\hat{u}\), and when this convergence extends to the stronger notion of pattern convergence as defined in \eqref{pattern convergence}.  We denote
\begin{align*}
    G_n(\theta) &= \frac{1}{n}\sum_{i=1}^n \ell(X_i, \theta), 
\end{align*}
and the empirical process associated with \(\mathcal{F}\) by
\[
    \nu_n \ell(\cdot,\theta)
    := n^{-1/2} \sum_{i=1}^n \big(\ell(X_i, \theta) - G(\theta)\big).
\]
We conclude this section by recalling several basic notions from convex analysis that will be used throughout the paper. 
For any set \(A \subseteq \mathbb{R}^p\), we denote by 
\(\operatorname{con}(A)\) its convex hull, 
\(\operatorname{span}(A)\) its linear span, 
\(\operatorname{par}(A) := \operatorname{span}(\{u - v : u, v \in A\})\) its parallel space, 
\(\operatorname{aff}(A) := x_0 + \operatorname{par}(A)\) its affine hull (for any \(x_0 \in A\)), 
and \(\operatorname{ri}(A)\) its relative interior, i.e., the interior of \(A\) within \(\operatorname{aff}(A)\) equipped with the subspace topology of \(\mathbb{R}^p\). 

For a convex function \(f : \mathbb{R}^p \to \mathbb{R}\), we define its directional derivative at \(x\) in direction \(u\) by
\[
    f'(x; u) := \lim_{\varepsilon \downarrow 0} \frac{f(x + \varepsilon u) - f(x)}{\varepsilon},
\]
which is itself a convex function of \(u\).
The subdifferential of \(f\) at \(x\) is given by
\[
    \partial f(x) := \{ v \in \mathbb{R}^p : f(y) \ge f(x) + \langle v, y - x \rangle \ \forall y \in \mathbb{R}^p \}.
\]

\end{subsection}

\section{Stochastic differentiability}

In this section, we develop the different notions of stochastic differentiability, summarized in Figure \ref{fig:diamond-graph}. We begin by formulating classical differentiability conditions on $\ell$. Let $B$ be an open neighborhood of $\theta_0$ such that the following regularity conditions hold:
\begin{align}\label{regularity conditions}
    \text{R1)}&\hspace{0.5cm} \text{the loss function}, \;\; \theta \mapsto \ell(x, \theta) \text{ is } C^2 \text{ on } B \text{ for all } x. \nonumber\\
    \text{R2)}&\hspace{0.5cm} \Vert \ell''(x, \theta) \Vert \leq M(x) \text{ for } \theta \in B, 
    \text{ where } M(x) \text{ satisfies } \mathbb{E}[M(X)^2] < \infty. \nonumber\\
    \text{R3)}&\hspace{0.5cm} \mathbb{E}[\ell'(X, \theta_0)\ell'(X, \theta_0)^{T}] < \infty.
\end{align}
The above conditions together with consistency and non-singularity of the Hessian of $G(\theta)$ at $\theta_0$ guarantee weak convergence of the rescaled error $\sqrt{n}(\hat{\theta}_n-\theta_0)$; see Theorem~\ref{main pattern robust theorem in distribution}. For a detailed discussion of these assumptions, we refer to Section~4 of \cite{pollard_1985}. Our goal now is to relax the twice-differentiability requirement.
We start with the assumption that the loss can be expanded as
\begin{equation}\label{Taylor expansion of loss}
\ell(x,\theta)=\ell(x,\theta_0)+(\theta-\theta_0)\triangle(x,\theta_0)+\Vert \theta-\theta_0\Vert r(x,\theta),
\end{equation}
where $x\mapsto\triangle(x, \theta_0)$ is a Borel measurable function, which does not depend on $\theta$, and $r(x, \theta)$ a rest term. If $\theta\mapsto\ell(x,\theta)$ is differentiable at $\theta_0$, \eqref{Taylor expansion of loss} is the first order Taylor expansion of $\ell(x,\theta)$ around $\theta_0$, with $\triangle(x,\theta_0)=\ell'(x,\theta_0)$.
%\subsection{Stochastic differentiability conditions}
Pollard \cite{pollard_1985} shows how the rather weak assumptions of twice differentiability of $G(\theta)$ together with a condition called stochastic differentiability %of $\ell(\cdot,\theta)$ at the minimizing value $\theta_0$ of $G(\theta)$ 
are sufficient for asymptotic normality of the minimizer of $G_n(\theta)$.

\begin{definition}
A function $\ell(\cdot,\theta)$ of the form (\ref{Taylor expansion of loss}) is called \textit{stochastically differentiable} at $\theta_0$ if for any sequence of balls $U_n$ around $\theta_0$ that is shrinking toward $\theta_0$: 
\begin{equation}\label{stochastic differentiability}
    \underset{\theta\in U_n}{\operatorname{sup}}\hspace{0,1cm} \dfrac{\vert\nu_n r(\cdot,\theta)\vert}{1+n^{1/2}\Vert\theta-\theta_0\Vert}\xrightarrow[n\rightarrow\infty]{p} 0,
\end{equation}
where $\nu_n r(\cdot,\theta)$ denotes the empirical process of the rest term.
\end{definition}

%\textcolor{red}{Gosia: what is $\nu_n r(\cdot,\theta)$ ? Above you defined only the empirical risk $\nu_n l(\cdot,\theta)$}. 
Paraphrasing David Pollard \cite{pollard_1985}, the most useful consequence of stochastic differentiability is that $\nu_n r(\cdot,\theta_n)=o_p(1+n^{1/2}\Vert\theta_n-\theta_0\Vert)$ for any sequence of random vectors $\theta_n$ converging in probability to $\theta_0$. In particular, if $\theta_n=\theta_0+O_p(n^{-1/2})$, then $\nu_n r(\cdot,\theta_n)=o_p(1+O_p(1))=o_p(1)$. This uniform control over the rest term, resulting from stochastic differentiability, is used to obtain asymptotic normality of (non-regularized) $M-$ estimators in Theorem 2 \cite{pollard_1985}.

Stochastic differentiability is a relatively weak condition. It is shown in \cite{pollard_1985} Section 4. that if there is an open neighborhood $B\subset\Theta$ around $\theta_0$ such that conditions R1,R2,R3 in \eqref{regularity conditions} hold, then $G(\theta)$ is twice differentiable at $\theta_0$ and $\ell(x,\theta)$ is stochastically differentiable \footnote{For stochastic differentiability, in R2) it suffices to assume existence of the first moment $\mathbb{E}[M(X)]<\infty$, see \cite{pollard_1985}. For the stronger SLD we assume $\mathbb{E}[M(X)^2]<\infty$.} at $\theta_0$. However, stochastic differentiability is a strictly weaker condition than regularity conditions (\ref{regularity conditions}). It does not require twice differentiability of the loss and holds, for example, for the Huber and the Median loss, see Section 5 in \cite{pollard_1985}. In Theorem~\ref{main pattern robust theorem in distribution}, which is a regularized extension of the central limit theorem in \cite{pollard_1985} Theorem 2., we replace stochastic differentiability with an even slightly weaker, local condition;
\begin{definition}
    We shall call a function $\ell(\cdot, \theta)$ of the form (\ref{Taylor expansion of loss}) \textit{locally stochastically differentiable} at $\theta_0$ if for every compact $K\subset \mathbb{R}^p$,
\begin{equation}\label{local stochastic differentiability}
    \underset{\theta\in \theta_0+K/\sqrt{n}}{\operatorname{sup}}\hspace{0,1cm} \vert\nu_n r(\cdot,\theta)\vert\xrightarrow[n\rightarrow\infty]{p} 0.
\end{equation}
\end{definition}
Stochastic differentiability (\ref{stochastic differentiability}) implies local stochastic differentiability  (\ref{local stochastic differentiability}) by setting $U_n=\theta_0+K/\sqrt{n}$, see appendix \ref{appendix local stochastic differentiability is weak}.

For the purpose of pattern convergence, we impose a slightly stronger regularity condition on the loss than the condition (\ref{local stochastic differentiability});

\begin{definition}\label{stochastic Lipschitz differentiability definition}
    We shall call a function $\ell(\cdot,\theta)$ of the form (\ref{Taylor expansion of loss}) \textit{stochastically Lipschitz differentiable} (SLD) at $\theta_0$ if for any compact $K\subset\mathbb{R}^p$, there is a sequence of random variables $L_n=o_p(1)$ such that $\tilde{r}(\cdot,\theta)=\Vert\theta-\theta_0\Vert r(\cdot,\theta)$ satisfies
\begin{equation}\label{stochastic Lipschitz differentiability condition}
\sup_{\overset{\theta_1,\theta_2\in\theta_0+K/\sqrt{n}}{\theta_1\neq\theta_2}}\dfrac{\vert\nu_n(\tilde{r}(\cdot,\theta_1)-\tilde{r}(\cdot,\theta_2))\vert }{\Vert\theta_1-\theta_2\Vert}\leq L_n,
\end{equation}
\end{definition}

We see that (\ref{stochastic Lipschitz differentiability condition}) implies (\ref{local stochastic differentiability}), because 
\begin{equation*}
\Vert\theta-\theta_0\Vert\cdot\sup_{\theta\in\theta_0+K/\sqrt{n}}\vert\nu_nr(\cdot,\theta)\vert=\sup_{\theta\in\theta_0+K/\sqrt{n}}\vert\nu_n(\tilde{r}(\cdot,\theta)-\tilde{r}(\cdot,\theta_0))\vert\leq L_n \Vert\theta-\theta_0\Vert.
\end{equation*}
The relations between regularity conditions (\ref{regularity conditions}) and the three notions of stochastic differentiability (\ref{stochastic differentiability}), (\ref{local stochastic differentiability}), (\ref{stochastic Lipschitz differentiability condition}) are summarized in Figure~\ref{fig:diamond-graph}.

\begin{remark}\label{remark sufficient Lipschitz condition}
If the map $\theta\mapsto\ell(x,\theta)$ is differentiable in a neighborhood $B$ around $\theta_0$ for every $x$ and 
\begin{equation*}
    \Vert\ell'(x,\theta_1)-\ell'(x,\theta_2)\Vert\leq M(x) \Vert \theta_1-\theta_2\Vert,
\end{equation*}
for $\theta_1,\theta_2\in B$ and some $M$ with finite second moment $\mathbb{E}[M(X)^2]<\infty$, then the SLD condition (\ref{stochastic Lipschitz differentiability condition}) holds. For proof of this fact, we refer the reader to appendix \ref{appendix stochastic Lipschitz differentiability for the Huber loss}. In particular, if the regularity conditions (\ref{regularity conditions}) hold, then 
\begin{align*}
    \Vert\ell'(x,\theta_1)-\ell'(x,\theta_2)\Vert&\leq \Vert \ell''(x,\theta)\Vert \Vert \theta_1-\theta_2\Vert\\
    &\leq M(x)  \Vert \theta_1-\theta_2\Vert,
\end{align*}
where $\mathbb{E}[M(X)^2]<\infty$. Consequently, regularity conditions in (\ref{regularity conditions}) imply SLD.
\end{remark}

The subsequent Lemma follows from similar arguments as the proof of Theorem 2 in \cite{pollard_1985}. 

\begin{lemma}\label{Jonas lemma} Assume $G(\theta)$ has a non-singular Hessian $C=\nabla^2\vert_{\theta=\theta_0}G(\theta)$ at its minimizing value $\theta_0$ and $\ell(\cdot,\theta)$ is locally stochastically differentiable at $\theta_0$ (i.e. \eqref{local stochastic differentiability} holds).  Then
\begin{equation*}
    n\left(G_n\left(\theta_0+\dfrac{u}{\sqrt{n}}\right)-G_n(\theta_0)\right)=\dfrac{1}{2}u^TCu+u^T\nu_n\triangle(\cdot, \theta_0)+R_n(u),
\end{equation*}
where $\sup_{u\in K}\vert R_n(u)\vert\rightarrow 0$ in probability as $n\rightarrow\infty$, for any compact set $K$. Further, if $\ell(\cdot,\theta)$ is stochastically Lipschitz differentiable at $\theta_0$ (i.e. (\ref{stochastic Lipschitz differentiability condition}) holds) and $\theta\mapsto\nabla^2_{\theta}G(\theta)$ exists around $\theta_0$ and is Lipschitz continuous at $\theta_0$, then for any compact $K$ we have $\vert R_n(u)-R_n(v)\vert\leq L_n\Vert u-v\Vert$ for $u,v\in K$, for some $L_n=o_p(1)$.
\end{lemma}
\begin{proof} Abbreviate $\theta=\theta_0+u/\sqrt{n}$. Because $G(\theta) = G_n(\theta) + n^{-1/2}\,\nu_n \ell(\cdot,\theta)$, we have 
\begin{equation*}
    n\left(G_n\left(\theta\right)-G_n(\theta_0)\right)= n\left(G\left(\theta\right)-G(\theta_0)\right)+n^{1/2}\nu_n\left[\ell\left(\cdot,\theta\right)-\ell(\cdot,\theta_0)\right].
\end{equation*}
We consider the two terms above separately. Taylor expansion of $G(\theta)$ around $\theta_0$ yields:
 \begin{align*}
     n(G(\theta)-G(\theta_0))&=n(\theta-\theta_0)^T C(\theta-\theta_0)/2+ nH(\theta) \Vert\theta-\theta_0\Vert^2 %o(\Vert\theta_n-\theta_0\Vert^2)]
     \\
     &=u^TCu/2 +  H(\theta_0+u/\sqrt{n})\Vert u\Vert ^2, %n\cdot o(\Vert u\Vert^2/n).
 \end{align*}
 where $H(\theta)\rightarrow 0$ as $\theta\rightarrow \theta_0$.
 At the same time
 \begin{align*}
     n^{1/2}\nu_n[\ell(\cdot,\theta)-\ell(\cdot,\theta_0)]&=n^{1/2}(\theta-\theta_0)^T \nu_n\triangle(\cdot,\theta_0)+n^{1/2}\Vert(\theta-\theta_0)\Vert \nu_n r(\cdot,\theta)\\
     &=u^T \nu_n\triangle(\cdot,\theta_0)+\Vert u\Vert \nu_n r(\cdot,\theta_0+ u/\sqrt{n}).
 \end{align*}
 Combining both expressions yields the desired representation with a rest term 
 \begin{equation*}
R_n(u)= \Vert u\Vert\cdot \nu_n r(\cdot,\theta_0+ u/\sqrt{n}) + H(\theta_0+u/\sqrt{n})\Vert u\Vert^2.
 \end{equation*}
 Stochastic differentiability guarantees that for any compact set $K$, $sup_{u\in K}\Vert R_n(u)\Vert\rightarrow 0$ in probability. %To show that $R_n(u)$ is Lispchitz, note that 
 If additionally (\ref{stochastic Lipschitz differentiability condition}) holds, using $\Vert u\Vert\cdot \nu_n r(\cdot,\theta_0+ u/\sqrt{n})=\sqrt{n}\cdot \nu_n \tilde{r}(\cdot,\theta_0+ u/\sqrt{n})$, we obtain
 \begin{align*}
     \left\vert \Vert u\Vert\cdot \nu_n r(\cdot,\theta_0+ u/\sqrt{n})-\Vert v\Vert\cdot \nu_n r(\cdot,\theta_0+ v/\sqrt{n})\right\vert \leq L_n\Vert u-v\Vert,
 \end{align*}
 with $L_n=o_p(1)$. Finally, because $\nabla^2_{\theta}G(\theta)$ exists around $\theta_0$ and is Lipschitz at $\theta_0$, the rest term $H(\theta)$ is Lipschitz continuous around $\theta_0$, see appendix \ref{appendix Lipschitzness of the Taylor rest term}. %For $G(\theta)\in C^3$ around $\theta_0$, 
 From this, one can directly verify that for any $u,v\in K$;
 \begin{align*}
     \vert H(\theta_0+u/\sqrt{n})\Vert u\Vert^2-H(\theta_0+v/\sqrt{n})\Vert v\Vert^2\vert \leq L_n' \Vert u-v\Vert,
 \end{align*}
 for some $L_n'=o_p(1)$. As a result
 \begin{equation*}
     \vert R_n(u)-R_n(v)\vert \leq (L_n+L'_n)\Vert u-v\Vert,
 \end{equation*}
 where $(L_n+L'_n)=o_p(1)$. This finishes the proof.
 \end{proof}

\section{Asymptotic distribution of patterns}

\begin{subsection}{Pattern convergence}
\label{sec:Pconv}
We focus on convex penalties of polyhedral type, that is, functions of the form
\begin{equation}\label{penalty form}
    f(\theta) = \max\{v_1^T\theta, \dots, v_k^T\theta\} + g(\theta),
\end{equation}
where $v_1, \dots, v_k \in \mathbb{R}^p$ are penalty-specific vectors, and $g(\theta)$ is a convex and differentiable function. 
This general form includes widely used regularizers such as the Lasso, Fused Lasso, SLOPE, and Elastic Net. 
The behavior of the model patterns induced by such penalties has been analyzed in detail in \cite{BOGDAN_pattern, graczyk2023pattern, hejny2025unveiling}.

Formally, the \emph{pattern} of $f$ at $\theta$ is defined as the set of indices that maximize the inner product $v_i^T\theta$:
\begin{equation*}
    I_f(\theta) := 
    \underset{i \in \{1, \dots, k\}}{\operatorname{argmax}} \{v_i^T\theta\}.
\end{equation*}
The set of all patterns $\mathfrak{P}_f := I_f(\mathbb{R}^p)$ is the image of the pattern map $I_f : \mathbb{R}^p \to \mathcal{P}(\{1, \dots, k\})$\footnote{$\mathcal{P}(\{1, \dots, k\})$ denotes the power set of $\{1, \dots, k\}$.}. 
We also define the \emph{pattern space} of $f$ at $\theta$ as 
\[
    \langle U_{\theta}\rangle_f = \operatorname{span}\{\tilde{\theta} \in \mathbb{R}^p : I_f(\tilde{\theta}) = I_f(\theta)\}.
\]
For notational simplicity, we omit the subscript $f$ when the penalty is clear from context and write 
$I(\theta) = I_f(\theta)$, 
$\langle U_{\theta}\rangle_f = \langle U_{\theta}\rangle$, 
and $\mathfrak{P}_f = \mathfrak{P}$.

As a familiar example, the Lasso penalty $\lambda \Vert \theta\Vert_1 = \lambda \sum_{i=1}^p \vert \theta_1\vert$ can be written in the form (\ref{penalty form}) as the maximizer over $k=2^p$ vectors $v_i$ of the form $(\pm\lambda,\dots,\pm\lambda)$. The pattern $I_f(\theta)$ can be identified with the sign vector $\mathrm{sgn}(\theta)= (\mathrm{sgn}(\theta_1), \dots, \mathrm{sgn}(\theta_p))$, in the sense that $I_f(\theta)=I_f(\tilde{\theta})\iff \mathrm{sgn}(\theta)=\mathrm{sgn}(\tilde{\theta})$. 
For general $f$, one can show that pattern can be equivalently defined by the subdifferential
\begin{equation*}
    I_f(\theta)=I_f(\tilde{\theta})\iff \partial f(\theta)=\partial f(\tilde{\theta}).
\end{equation*}

\begin{definition}
    Let $(\hat{u}_n)_{n\in\mathbb{N}}$ be a sequence of random vectors in $\mathbb{R}^p$. We say that $\hat{u}_n$ converges \textit{in pattern} to a random vector $\hat{u}$ if 
    \begin{equation}\label{pattern convergence}
        \lim_{n\rightarrow \infty}\mathbb{P}[I_f(\hat{u}_n)=\mathfrak{p}]=\mathbb{P}[I_f(\hat{u})=\mathfrak{p}],
    \end{equation}
    for every pattern $\mathfrak{p}\in\mathfrak{P}_f$.
\end{definition}
We note that due to the discontinuous nature of the pattern function $I_f$, pattern convergence does not follow from convergence in distribution. 
\end{subsection}
\begin{subsection}{Main results}
The following theorems are the main results of the paper establishing the asymptotic distribution of the error (Theorem~\ref{main pattern robust theorem in distribution}) and its pattern (Theorem~\ref{main pattern robust theorem}) respectively for $\hat{u}_n=\sqrt{n}(\hat{\theta}_n-\theta_0)$ for \eqref{main objective}. Theorem~\ref{main pattern robust theorem in distribution} is a penalized analog of Theorem~2 \cite{pollard_1985} and generalizes Theorem 2~\cite{fu2000asymptotics}. We note that Theorem~2 in \cite{pollard_1985} assumes consistency of the estimator. In our case, we replace this assumption with uniform tightness (condition $iv)$) and the existence of a uniform envelope (condition $v)$). We formulate the following regularity conditions:
    \begin{align}\label{conditions on the loss}
    i) &\hspace{0,2cm} \ell(x,\theta) \text{ is locally stochastically differentiable at } \theta_0, \text{ (i.e. \eqref{local stochastic differentiability} holds).}\nonumber \\
    ii) &\hspace{0.3cm} G(\theta) \text{ has a positive definite Hessian } C=\nabla^2_{\theta=\theta_0}G(\theta), \text{ at its minimizing value } \theta_0. \nonumber\\
    iii) &\hspace{0.2cm} \text{$\mathbb{E}[\triangle(x,\theta_0)]=0$ and the covariance matrix } C_{\triangle} =\mathbb{E}[\triangle(x,\theta_0)\triangle(x,\theta_0)^T]<\infty. \\
    iv) &\hspace{0.2cm} %\theta_0 \text{ is an interior point of } \Theta \text{ and }
    \hat{\theta}_n \text{ is uniformly tight}.\nonumber\\
    v) &\hspace{0.2cm} \text{For every compact $K\subset\Theta$; $\sup_{\theta\in K}|\ell(x,\theta)|\leq L(x)$ for some $L$ with $\mathbb{E}[|L(X)|]<\infty$.}\nonumber
\end{align}
\begin{theorem}\label{main pattern robust theorem in distribution}
    Let $f$ be a penalty of the form \eqref{penalty form} and assume that $i)-v)$ hold. Then $\hat{u}_n=\sqrt{n}(\hat{\theta}_n-\theta_0)$ converges in distribution to the random vector $\hat{u}:=\textup{argmin}_{u\in\mathbb{R}^p} V(u)$, where
\begin{equation}\label{main objective V(u)}
    V(u) = \dfrac{1}{2}u^{T}Cu-u^{T}W+f'({\theta_0};u),
\end{equation}
and $W\sim\mathcal{N}(0,C_{\triangle})$.
\end{theorem}
In order to guarantee pattern convergence, we impose the stronger SLD at $\theta_0$, and a Lipschitzness of the Hessian of $G(\theta)$ around $\theta_0$:
\begin{align}
    i') &\hspace{0,2cm} \ell(x,\theta) \text{ is stochastically Lipschitz differentiable at } \theta_0, \text{ (i.e. \eqref{stochastic Lipschitz differentiability condition} holds),} \nonumber\\
    ii') &\hspace{0.3cm} \nabla^2_{\theta}G(\theta) \text{ exists around } \theta_0 \text{ and is Lipschitz continuous at } \theta_0, 
\end{align}
\begin{theorem}\label{main pattern robust theorem}
Let $f$ be a penalty of the form \eqref{penalty form} and assume that additionally to $i)-v)$, $i'),ii')$ also hold. Then $\hat{u}_n$ converges in distribution and in pattern to $\hat{u}$ given by \eqref{main objective V(u)} in the sense \eqref{pattern convergence}; 
\begin{equation*}
        \lim_{n\rightarrow \infty}\mathbb{P}[I_f(\hat{u}_n)=\mathfrak{p}]=\mathbb{P}[I_f(\hat{u})=\mathfrak{p}],
    \end{equation*}
    for every pattern $\mathfrak{p}\in\mathfrak{P}_f$.
\end{theorem}
\begin{proof}[Proof of Theorem~\ref{main pattern robust theorem in distribution}]
By definition, $\hat{\theta}_n$ is the minimizer of 
\begin{equation}\label{M_n}
M_n(\theta):=G_n(\theta)+n^{-1/2}f(\theta).
\end{equation}
The rescaled error $\hat{u}_n=\sqrt{n}(\hat{\theta}_n-\theta_0)$ minimizes
\begin{align}\label{V_n(u)}
    V_n(u)&=  n\left(G_n\left(\theta_0+\dfrac{u}{\sqrt{n}}\right)-G_n(\theta_0)\right)+\sqrt{n}\left(f\left(\theta_0+u/\sqrt{n}\right)-f(\theta_0)\right).
\end{align}
To show weak convergence of $\hat{u}_n$ to $\hat{u}$, we shall apply the argmax theorem; see Theorem 3.2.2 in \cite{van1998asymptotic}. For this we need to verify:
\begin{align*}
    i)&\hspace{0,2cm} \hat{\theta}_n \text{ is consistent for }\theta_0, \\
    ii)&\hspace{0,2cm} \hat{u}_n=\sqrt{n}(\hat{\theta}_n-\theta_0) \text{ is uniformly tight},\\
    iii)&\hspace{0,2cm} V_n \text{ converges weakly as a process to } V \text{ in } \ell^{\infty}(K) \text{ for every compact } K\subset\mathbb{R}^p. 
\end{align*}
Here, $\ell^{\infty}(K)$ denotes the metric space of bounded functions from $K$ to $\mathbb{R}$, equipped with the supremum metric $d(f,g):=\sup_{u\in K}\vert f(u)-g(u) \vert$.

\textbf{Step 1:} (consistency of $\hat{\theta}_n$) \\
Consider a compact set $K \subset \Theta$. By $(v)$, measurability of the maps $x\mapsto\ell(x,\theta)$ and continuity of the maps $\theta\mapsto\ell(x,\theta)$, the family $\mathcal{F}_K=\{\ell(x,\theta):\theta\in K\}$ is Glivenko-Cantelli, i.e. $$\sup_{\theta\in K}|G_n(\theta)-G(\theta)|\overset{p}{\longrightarrow} 0,$$
see Lemma~2.4 in \cite{newey1994large} or Example~19.8 in \cite{van1998asymptotic}. Hence $\sup_{\theta\in K}|M_n(\theta)-G(\theta)|\overset{p}{\longrightarrow} 0$. Since $G(\theta)$ is continuous, has a unique minimum at $\theta_0$, and $\hat{\theta}_n$ is uniformly tight by assumption, Corollary 3.2.3 \cite{vanderVaart1996} implies that $\hat{\theta}_n\rightarrow \theta_0$ in probability.
\vspace{0.3cm}

\textbf{Step 2:} (uniform tightness of $\hat{u}_n=\sqrt{n}(\hat{\theta}_n-\theta_0)$) \\
Next, we establish the rate of convergence as $\hat{\theta}_n-\theta_0=O_p(n^{-1/2})$. We verify the conditions of Theorem 3.2.5 \cite{vanderVaart1996}, to obtain uniform tightness of $\hat{u}_n=\sqrt{n}(\hat{\theta}_n-\theta_0)$. We have 
\begin{equation*}
    %M(\theta)-M(\theta_0)=
    G(\theta)-G(\theta_0)=(\theta-\theta_0)^TC(\theta-\theta_0)/2\gtrsim\Vert\theta-\theta_0\Vert^2,
\end{equation*}
since $C$ is non-singular \footnote{The notation $\gtrsim$ and $\lesssim$ stands for bounded from below resp. from above up to a universal constant. }. Also, for any $\delta>0$ and $\Vert\theta-\theta_0\Vert\leq\delta$, %$M_n=G_n+n^{-1/2}f, M=G$
\begin{align*}
    \vert (G_n-G)(\theta)-(G_n-G)(\theta_0)\vert&=n^{-1/2}\vert\nu_n(\ell(\cdot,\theta)-\ell(\cdot,\theta_0))\vert\\
    &\leq n^{-1/2}\Vert\theta-\theta_0\Vert(\Vert\nu_n\triangle(\cdot,\theta_0)\Vert+\vert\nu_n r(\cdot,\theta)\vert)\\
    &\lesssim n^{-1/2} \delta \cdot O_p(1)
\end{align*}
And recall that $f(\theta)=\max\{v_i^T\theta:1\leq i \leq k\}+g(\theta)$ as in (\ref{penalty form}), so that for $\Vert\theta-\theta_0\Vert\leq\delta$,
\begin{align*}
    n^{-1/2}(f(\theta)-f(\theta_0))\leq n^{-1/2}\Vert\theta-\theta_0\Vert(\max_{1\leq i \leq k}\Vert v_i\Vert+\Vert\nabla g(\theta^*)\Vert) \lesssim n^{-1/2} \delta 
\end{align*}
for some $\theta^*$ on the line segment between $\theta$ and $\theta_0$. Hence for sufficiently small $\delta>0$:
\begin{equation*}
\mathbb{E}\left(\sup_{\theta \in B_{\delta}(\theta_0)}\vert(M_n-G)(\theta)-(M_n-G)(\theta_0)\vert\right)\lesssim n^{-1/2}\delta.%\vert (G_n-G)(\theta)-(G_n-G)(\theta_0) + n^{-1/2}(f(\theta)-f(\theta_0))\vert 
\end{equation*}
%where we are using $M_n=G_n+n^{-1/2}f, M=G$.
Now because $\hat{\theta}_n\rightarrow \theta_0$, Theorem 3.2.5 \cite{vanderVaart1996} gives uniform tightness of $\hat{u}_n=\sqrt{n}(\hat{\theta}_n-\theta_0)$, i.e. $\sqrt{n}\Vert\hat{\theta}_n-\theta_0\Vert=O_p(1)$.
\vspace{0.3cm}

\textbf{Step 3:} (Weak convergence of $V_n(u)$ in $\ell^{\infty}(K)$)\\
Rewriting $V_n(u)$ in \eqref{V_n(u)} by Lemma~\ref{Jonas lemma}, we obtain:
\begin{equation*}
    V_n(u)=\dfrac{1}{2}u^TCu+u^T W_n+R_n(u) + \sqrt{n}\left(f\left(\theta_0+u/\sqrt{n}\right)-f(\theta_0)\right),
\end{equation*}
with $W_n:=\nu_n\triangle(\cdot, \theta_0)$, and $\sup_{u\in K}\|R_n(u)\|\overset{p}{\rightarrow} 0$. By the central limit theorem, $W_n$ converges weakly to the random vector $W\sim\mathcal{N}(0,C_{\triangle})$, where $C_{\triangle}=\mathbb{E}[\triangle(x,\theta_0)\triangle(x,\theta_0)^T]$. Now consider the mapping 
\begin{equation*}
    T:w\mapsto \frac12u^TCu + u^Tw. 
\end{equation*}
We see that $T$ is a continuous map between $(\mathbb{R}^p, \|\cdot\|_2)$ and $(\ell^{\infty}(K),\|\cdot\|_K)$, since
\begin{equation*}
    \|T(w)-T(\tilde{w})\|_K=\sup_{u\in K}|u^T(w-\tilde{w})|\leq \sup_{u\in K}(\|u\|_2) \|w-\tilde{w}\|_2.
\end{equation*}
Therefore, by the continuous mapping theorem (Theorem 18.11 \cite{van1998asymptotic}), $T(W_n)$ converges weakly to $T(W)$ in $\ell^{\infty}(K)$.
Finally, $\sup_{u\in K}\|R_n(u)\|\overset{p}{\rightarrow} 0$ and  $\sqrt{n}\left(f\left(\theta_0+u/\sqrt{n}\right)-f(\theta_0)\right)$ converges to the deterministic limit $f'(\theta_0;u)$ in $\ell^{\infty}(K)$. From the Slutsky's theorem (Theorem 18.10 \cite{van1998asymptotic}) we conclude that $V_n(u)$ converges weakly to $V(u)$ in \eqref{main objective V(u)} as a stochastic process in $\ell^{\infty}(K)$. Putting all pieces together, we have established uniform tightness of $\hat{u}_n$ and weak convergence of $V_n$ to $V$ in $\ell^{\infty}(K)$. Moreover, all sample paths of $u\mapsto V(u)$ are continuous and  possess a unique minimum by positive definiteness of $C$. As a result, Theorem 3.2.2 \cite{vanderVaart1996} implies that $\hat{u}_n$ converges weakly to $\hat{u}$ in $\mathbb{R}^p$. 
\vspace{0.3cm}

\end{proof}

We now turn to the proof of Theorem~\ref{main pattern robust theorem}. In this proof we assume knowledge about subdifferentials, for details see \cite{hiriart2013convex, rockafellar2009variational}.
\begin{proof}[Proof of Theorem~\ref{main pattern robust theorem}]
%\textbf{Step 4:} \\
It remains to show that under the stronger regularity assumptions $(i')-(ii')$, the pattern $I(\hat{u}_n)$ converges weakly to $I(\hat{u})$. By Lemma \ref{Jonas lemma}, we can express
\begin{equation*}
    V_n(u)=V(u)+R_n'(u),
\end{equation*}
where
\begin{align*}
    R'_n(u)=R_n(u) + u^T(\nu_n\triangle(\cdot,\theta_0)-W)+\sqrt{n}\left(f\left(\theta_0+u/\sqrt{n}\right)-f(\theta_0)\right)-f'(\theta_0;u),
\end{align*}
so that for every compact set $K$; $\sup_{u\in K}\vert R'_n(u)\vert\rightarrow 0$ and $\vert R'_n(u)-R'_n(v)\vert\leq L'_n \Vert u-v\Vert$ for $u,v\in K$, with $L'_n=o_p(1)$. Taking the subdifferential of (\ref{main objective V(u)}) gives:
\begin{equation*}
    \partial V(u) = Cu-W+\partial_u f'(\theta_0;u).
\end{equation*}

The solution $\hat{u}$ minimizes (\ref{main objective V(u)}) and solves the optimality condition $0\in\partial V(\hat{u})$, or equivalently, by the above \footnote{Here, $\partial f'(\theta_0;\hat{u})$ stands for the subdifferential of the convex function $u\mapsto f'(\theta_0;u)$ at $\hat{u}$.}
    \begin{equation*}
        W\in C\hat{u} + \partial f'(\theta_0;\hat{u}).
    \end{equation*}
    
With high probability, $W$ will fall into the relative $\delta$ interior of $C\hat{u} + \partial f'(\theta_0;\hat{u})$. Precisely, by Lemma~\ref{lemma on subdifferential variational property} with $h(u)=f'(\theta_0;u)$, for any $\varepsilon>0$ there is $\delta>0$ s.t. 
\begin{equation}\label{margin for the W}
    \mathbb{P}[B^{*}_{\delta}(0)\subset C\hat{u} -W + \partial f'(\theta_0;\hat{u})]\geq 1-\varepsilon,
\end{equation}
where $B^{*}_{\delta}(0)=\{u\in \mathbb{R}^p :  \|u\|_2<\delta\}\cap\mathrm{par}(\partial f'(\theta_0;\hat{u}))$ is the relative $\delta-$ball in the parallel space $\mathrm{par}(\partial f'(\theta_0;\hat{u}))$. %=\mathrm{span}\{v-w:v,w\in \partial f'(\theta_0;\hat{u})\}$.
Furthermore, because $L'_n=o_p(1)$, the event 
\begin{align}
    \delta/2-L'_n&>0,\label{high prob events}
\end{align}
occurs with high probability as $n$ increases. Finally, $\hat{u}_n\rightarrow \hat{u}$ in distribution and, by the Skorokhod representation theorem there exist versions $\tilde{u}_n\overset{(d)}{=}\hat{u}_n$ for each $n$ and $\tilde{u}\overset{(d)}{=}\hat{u}$, such that $\tilde{u}_n$ converges to $\tilde{u}$ almost surely. For notational simplicity, we can assume without loss of generality that 
\begin{equation}
    \hat{u}_n\overset{a.s.}{\longrightarrow} \hat{u}.
\end{equation} 
Therefore, for all sufficiently large $n$:
\begin{align}\label{estimation error bound}
    \Vert \hat{u}_n-\hat{u}\Vert&\leq \delta\Vert
    C\Vert_{op}^{-1}/2
\end{align}

\vspace{0.3cm}

%\textbf{Step 5:} \\
We shall now show that the perturbation $R'_n(u)$ in $V_n(u)=V(u)+R'_n(u)$ does not break the pattern of $\hat{u}$. Let $P$ denote the projection onto the pattern space of $\hat{u}$, which is given by $\langle U_{\hat{u}}\rangle=span\{u: I(u)=I(\hat{u})\}$. By convexity of $V(u)$, we have 
\begin{equation*}
    V(\hat{u}_n)\geq V(P\hat{u}_n)+\langle v, \hat{u}_n-P\hat{u}_n \rangle,
\end{equation*}
where $v$ is any subgradient vector in 
$$\partial V(P\hat{u}_n)=CP\hat{u}_n-W+\partial f'(\theta_0;P\hat{u}_n).$$ 
Notice that by definition $P\hat{u}_n$ lies in the pattern space of $\hat{u}$. Also, because $\hat{u}_n$ converges to $\hat{u}$ almost surely, we eventually have $I(P\hat{u}_n)=I(\hat{u})$, and consequently, also \begin{equation}\label{key subdifferential equation in Step 5}
    \partial f'(\theta_0;P\hat{u}_n)=\partial f'(\theta_0;\hat{u}).
\end{equation}
Now assume that the high probability events in (\ref{margin for the W}) and (\ref{high prob events}) occur. Then combining the above observations yields:% with probability at least $1-\varepsilon$, 
\begin{align*}
    V(\hat{u}_n)-V(P\hat{u}_n)&\geq   \sup_{w\in\partial f'(\theta_0;P\hat{u}_n)}\langle CP\hat{u}_n-W+w,\hat{u}_n-P\hat{u}_n \rangle\\
    & = \sup_{w\in\partial f'(\theta_0;\hat{u})}\langle C\hat{u}-W+w,\hat{u}_n-P\hat{u}_n \rangle +\langle C(P\hat{u}_n-\hat{u}),\hat{u}_n-P\hat{u}_n\rangle\\
    & \geq  \sup_{v\in B^*_{\delta}(0)}\langle v,\hat{u}_n-P\hat{u}_n \rangle +\langle C(P\hat{u}_n-\hat{u}), \hat{u}_n-P\hat{u}_n\rangle\\
    & \geq  \delta\Vert \hat{u}_n-P\hat{u}_n\Vert - \Vert P\hat{u}_n-\hat{u}\Vert\cdot\Vert C\Vert_{op}\cdot \Vert \hat{u}_n-P\hat{u}_n\Vert\\
    &\geq \dfrac{\delta}{2}\Vert \hat{u}_n-P\hat{u}_n\Vert,
\end{align*}
where the last inequality holds because $\Vert P\hat{u}_n-\hat{u}\Vert\leq\Vert\hat{u}_n-\hat{u}\Vert\leq \delta\Vert C\Vert^{-1}_{op}/2$, by \eqref{estimation error bound}. Consequently, since $\hat{u}_n$ minimizes $V_n = V+R'_n$;% and by Lemma \ref{Jonas lemma}:
\begin{align*}
    0\geq V(\hat{u}_n)+R'_n(\hat{u}_n) - (V(P\hat{u}_n) +R'_n(P\hat{u}_n))\geq \left(\delta/2 - L'_n \right) \Vert \hat{u}_n-P\hat{u}_n\Vert.
\end{align*}
For $\delta/2-L'_n>0$, it follows that $\hat{u}_n= P\hat{u}_n$ and hence $I(\hat{u}_n)=I(P\hat{u}_n)=I(\hat{u})$. This proves that $\mathbb{P}[I(\hat{u}_n)= I(\hat{u})]\rightarrow 1$, as $n\rightarrow\infty$.
\end{proof}

We now formulate a weaker result, which follows directly from Theorem~\ref{main pattern robust theorem}, in terms of more standard regularity assumptions.  
\begin{corollary}\label{corollary of the main pattern theorem}
    Let $f$ be a penalty of the form \eqref{penalty form} and denote  by $\theta_0$ the minimizing value of $G(\theta)=\mathbb{E}[\ell(X,\theta)]$. Assume there is an open neighborhood $U$ around $\theta_0$, on which $\theta\mapsto\ell(x,\theta)$ is $C^2$  for every x, and:
    \begin{align*}
    i) &\hspace{0,2cm}  \|\ell''(x,\theta)\Vert\leq M(x) \text{ for } \theta\in U, \text{ for }M(x) \text{ with } \mathbb{E}[M(X)^2]<\infty. \\
    ii) &\hspace{0.3cm} G(\theta)\text{ is }C^3 \text{ on } U \text{ and has a positive definite Hessian } C=\nabla^2_{\theta=\theta_0} G(\theta) \text{ at } \theta_0. \\
    iii) &\hspace{0.2cm} \text{ $\mathbb{E}[\ell'(x,\theta_0)]=0$ and the covariance matrix } C_{\triangle} =\mathbb{E}[\ell'(x, \theta_0)\ell'(x, \theta_0)^T]<\infty. \\
    iv) &\hspace{0.2cm} \theta_0 \text{ is an interior point of } \Theta \text{ and }\hat{\theta}_n \text{ is uniformly tight}.\\
    v) &\hspace{0.2cm} \text{For every compact $K\subset\Theta$; $\sup$$_{\theta\in K}|\ell(x,\theta)|\leq L(x)$ for some $L$ with $\mathbb{E}[|L(X)|]<\infty$.}
\end{align*}
     Then $\hat{u}_n=\sqrt{n}(\hat{\theta}_n-\theta_0)$ converges in distribution to the minimizer $\hat{u}$ of \eqref{main objective V(u)}. Moreover, $\hat{u}_n$ converges to $\hat{u}$ in pattern in the sense \eqref{pattern convergence}.
\end{corollary}
\begin{proof}
    The proof follows directly from Theorem~\ref{main pattern robust theorem}. By remark \ref{remark sufficient Lipschitz condition}, condition $i)$ (corresponding to \eqref{regularity conditions}) implies Lipschitz stochastic differentiability, and thus also stochastic differentiability at $\theta_0$. Condition $ii)$ is strengthened to include $ii')$, hence all assumptions in Theorem~\ref{main pattern robust theorem} are satisfied, and the statement follows.
\end{proof}
\end{subsection}

\section{Pattern recovery}
By Theorem~\ref{main pattern robust theorem}, the limiting distribution of $\hat{u}_n=\sqrt{n}(\hat{\theta}_n-\theta_0)$ and its pattern $I(\hat{u}_n)$ is fully determined by the limiting random vector $\hat{u}$ given by \eqref{main objective V(u)}. This double convergence toward a limit of the form \eqref{main objective V(u)} is exactly the Assumption~A in Section 3 \cite{hejny2025unveiling}. That assumption has several important and direct implications for the limiting behaviour of $I(\hat{\theta}_n)$. For completeness, we restate these implications explicitly:
\vspace{0.3cm}

\textbf{Implication 1:} (Corollary 3.4 \cite{hejny2025unveiling}) \\
The limiting distribution of the pattern $I(\hat{\theta}_n)$ is completely determined by

\begin{equation*}
    \mathbb{P}[I(\hat{\theta}_n) =\mathfrak{p} ]\underset{n\to\infty}{\longrightarrow} \mathbb{P}\left[I_{\theta_0}(\hat{u}) =\mathfrak{p}\right],%,
\end{equation*}
where
\begin{equation*}
    I_{\theta_0}(\hat{u}):=\lim_{\varepsilon\downarrow 0}I(\theta_0+\varepsilon \hat{u})
\end{equation*}
and $\hat{u}$ is the minimizer of (\ref{main objective V(u)}).
\vspace{0.3cm}

\textbf{Implication 2:} (Theorem 3.5 \cite{hejny2025unveiling})\\
Moreover, the probability of recovering the true pattern is given by
\begin{equation*}
    \mathbb{P}[I(\hat{\theta}_n)=I(\theta_0)]\longrightarrow \mathbb{P}[\mathcal{N}(\mu,\Sigma)\in\partial f(\theta_0)],
\end{equation*}
where 
\begin{align*}
    \mu &=C^{1/2} P C^{-1/2}v_0\\
    \Sigma &= C^{1/2}(I-P)C^{1/2}C_{\triangle}C^{1/2}(I-P)C^{1/2},
\end{align*}
with $v_0$ any vector in $\partial f(\theta_0)$, and $P$ the projection on $C^{1/2}\langle U_{\theta_0}\rangle$. Here  $\langle U_{\theta_0}\rangle=span\{\theta: I(\theta)=I(\theta_0)\}$ denotes the pattern space at $\theta_0$, and $C$ and $C_{\triangle}$  are as in Theorem~\ref{main pattern robust theorem}.
\vspace{0.3cm}

\textbf{Implication 3:} (Corollary 3.7 \cite{hejny2025unveiling}) \\
If the asymptotic irrepresentability condition 
$\mu \in ri(\partial f(\theta_0))$ holds, or equivalently $C\langle U_{\theta_0}\rangle\cap ri(\partial f(\theta_0))\neq \emptyset$, then adding an $\alpha$ scaling to the penalty $n^{-1/2}\alpha f(\theta)$ in (\ref{main objective}), for all $\alpha\geq 0$;
\begin{align*}
\lim_{n\rightarrow\infty}\mathbb{P}\big[I(\hat{\theta}_n)=I(\theta_0)\big]\geq 1- 2e^{-c\alpha^2},
\end{align*}
for some constant $c>0$.%; see \cite{hejny2025unveiling}.
\vspace{0.3cm}

\textbf{Implication 4:} (Theorem 3.10 \cite{hejny2025unveiling}) \\

If $\hat{\theta}_n^{(1)}$ is any initial sequence of estimators such that $\sqrt{n}(\hat{\theta}_n^{(1)}-\theta_0)\overset{d}{\longrightarrow} W$ for some random vector $W$ with sub-gaussian entries. Then for 
\begin{equation*}
    \hat{\theta}_n^{(2)}=\text{Prox}_{n^{-1/2}\alpha f}(\hat{\theta}_n^{(1)})=\underset{\theta\in\mathbb{R}^p}{\operatorname{argmin}}\hspace{0.2cm}\frac{1}{2}\Vert \hat{\theta}_n^{(1)}-\theta\Vert_2^2+n^{-1/2}\alpha f(\theta).
\end{equation*}
for all $\alpha\geq0$;
\begin{align*}
\lim_{n\rightarrow\infty}\mathbb{P}\big[I(\hat{\theta}_n^{(2)})=I(\theta_0)\big]\geq 1- 2e^{-c\alpha^2},
\end{align*}
for some constant $c>0$, provided that $\langle U_{\theta_0}\rangle\cap ri(\partial f(\theta_0))\neq \emptyset$. In particular, this is true for any $\theta_0$ and for any penalty $f$ given by a symmetric norm such as Lasso, SLOPE, or Elastic Net. 

\subsection{Vague clustering}
Pattern convergence as defined in (\ref{pattern convergence}) is sensitive to infinitesimal perturbations. A more refined way of measuring how well $\hat{\theta}_n$ is recovering the pattern of $\theta_0$ is to consider the error $\hat{u}_n=\sqrt{n}(\hat{\theta}_n-\theta_0)$ and decompose it into its pattern error and residual error component by projecting onto the pattern space $\langle U_{\theta_0}\rangle$. Let $P_{\theta_0}$ be the projection onto the pattern space of $\theta_0$, then we can decompose the error as:
\begin{align*}
    \Vert \hat{u}_n\Vert_2^2=\underbrace{\Vert P_{\theta_0}\hat{u}_n\Vert_2^2}_{\text{pattern error}}+\underbrace{\Vert (I-P_{\theta_0})\hat{u}_n\Vert_2^2}_{\text{residual error}}.
\end{align*}
\begin{example}[Lasso / support pattern]
Let \(\theta_0=(1,-2,0,0)^\top\) and consider the Lasso, whose pattern is the active set \(S=\{i:\theta_{0,i}\neq 0\}=\{1,2\}\). Here the pattern subspace is
\(
T_{\theta_0}=\{v\in\mathbb{R}^4:\ v_{S^c}=0\}.
\)
The orthogonal projector is \(P_{\theta_0}=\mathrm{diag}(1,1,0,0)\). Hence, for any \(u=(u_1,u_2,u_3,u_4)^\top\),
\[
P_{\theta_0}u=(u_1,u_2,0,0)^\top,
\qquad
(I-P_{\theta_0})u=(0,0,u_3,u_4)^\top.
\]
In particular, the pattern error is \(\|P_{\theta_0}\hat u_n\|_2^2=\hat u_{n,1}^2+\hat u_{n,2}^2\) and the residual error equals \(\|(I-P_{\theta_0})\hat u_n\|_2^2=\hat u_{n,3}^2+\hat u_{n,4}^2\).
\end{example}

The residual error $RE(\hat{\theta}_n):=\Vert (I-P_{\theta_0})\hat{u}_n\Vert_2^2$ is a measure for how well $\hat{\theta}_n$ recovers the pattern of $\theta_0$, or more precisely, its pattern space $\langle U_{\theta_0}\rangle$. For sufficiently large $n$, exact pattern recovery $I(\hat{\theta}_n)=I(\theta_0)$ is equivalent to the residual error vanishing $\Vert (I-P_{\theta_0})\hat{u}_n\Vert_2^2=0$. Alternatively, we can consider the relative residual error
\begin{align}\label{eq:RRE}
    RRE(\hat{\theta}_n) = \frac{\Vert (I-P_{\theta_0})\hat{u}_n\Vert_2^2}{\Vert \hat{u}_n\Vert_2^2}= 1-\frac{\Vert P_{\theta_0}\hat{u}_n\Vert_2^2}{\Vert \hat{u}_n\Vert_2^2},
\end{align}
which measures the proportion of the error due to misspecifying the pattern compared to the total estimation error. Note that $REE(\hat{\theta}_n)$ is between zero and one. It is small when the estimator is relatively close to the pattern space of $\theta_0$, and exactly $0$, whenever $I(\hat{\theta}_n)=I(\theta_0)$. Finally, we remark that $RE(\hat{\theta}_n)$ and $REE(\hat{\theta}_n)$ can be viewed as the residual sum of squares and $1-R^2$ in the linear model $\hat{u}_n=U_M \beta + \varepsilon$, where the design matrix $U_M$ is a matrix whose columns form a basis of the pattern space $\langle U_{\theta_0}\rangle$.
\section{Examples}
Assume a regression model $y=f_{\theta}(x)+\varepsilon$, where $\{f_{\theta}:\theta\in\Theta\}$ is some parametrized family of regression functions and $\varepsilon$ some random noise. We estimate $\theta$ from observed i.i.d. pairs of data $(y_i, x_i)\in\mathbb{R}\times\mathbb{R}^p$, by solving a penalized optimization problem as in (\ref{main objective}):
\begin{equation}\label{main objective for regression}
     \hat{\theta}_n = \underset{\theta\in\Theta}{\operatorname{argmin}}  \dfrac{1}{n}\sum_{i=1}^{n} \ell(y_i, x_i, \theta) + n^{-1/2}f(\theta),
\end{equation}
where $\{\ell(y,x,\theta):\theta\in\Theta\}$ is a suitable family of loss functions. We note that the data $(y_i,x_i)$ now play the role of the observations $X_i$.
\subsection{Exponential family regression models}
Suppose $y_1,y_2,\dots,y_n$ are independent identically distributed response variables coming from an exponential density
\begin{equation}
  y_i\sim g_{\eta_i}(y_i)=e^{\tau(y_i\eta_i - \psi(\eta_i))}g_0(y_i),
\end{equation}
where $\eta_i=x_i^T \theta\in\mathbb{R}$, with $\theta\in \Theta=\mathbb{R}^p$; see \cite{Efron_2022}. If $\theta_0$ is the true parameter, the regression function is given by $\mathbb{E}[y\vert x]=\dot{\psi}(\eta)=\dot{\psi}(x^T\theta_0)$ and the model can be written as
\begin{equation*}
    y_i=\dot{\psi}(x_i^T\theta_0)+\varepsilon_i,
\end{equation*}
with noise $\varepsilon_i$ which might depend on $x_i$, but $\mathbb{E}[\varepsilon_i\vert x_i]=0$. The negative log likelihood is:
\begin{align*}
    -log (g_{\boldsymbol{\eta}}(y))&= -log \prod_{i=1}^n g_{\eta_i}(y_i)=\tau\sum_{i=1}^n \underbrace{\psi(\eta_i)-y_i \eta_i}_{\ell(y_i,\eta_i)} + \sum_{i=1}^n log(g_0(y_i)).
\end{align*}
Substituting $\eta_i = x_i^T\theta$ we obtain the loss function $\ell(y, x,\theta) = \psi (x^T\theta)-y x^T \theta$, which can be used to obtain an $M-$ estimate of $\theta_0$, given by (\ref{main objective}). Without penalization, this corresponds to the MLE estimator minimizing the negative log-likelihood. Note that the observed data in this context are the paired i.i.d observations $(y_i, x_i)$.  We now verify that under mild conditions on $\psi$, Theorem \ref{main pattern robust theorem} applies for generalized linear models. 
\begin{corollary}
Let $\psi$ be a strictly convex $C^2$ function. Assume the random vector $x\in\mathbb{R}^p$ has positive definite covariance $\mathbb{E}[xx^T]$, and that $G(\theta)$ is $C^3$ on some open neighborhood $U$ around $\theta_0$, such that $\mathbb{E}[\Vert xx^Tsup_{\theta\in U}\ddot{\psi}(x^T\theta)\Vert^2]<\infty$. Finally, suppose $\mathbb{E}[\sup_{\theta\in K}\psi(x^T\theta)]<\infty$ for every compact set $K$. If $\hat{\theta}_n$ is uniformly tight, then $\sqrt{n}(\hat{\theta}_n-\theta_0)$ converges in distribution and in pattern to $\hat{u}$ given by \eqref{main objective V(u)}, %in Theorem \ref{main pattern robust theorem}
 with 
$C= \mathbb{E}[xx^T\ddot{\psi}(x^T\theta_0)]$, $C_{\triangle}=\tau^{-1}C$.
    
\end{corollary}
 \begin{proof}
Differentiating twice we compute:
\begin{align*}
    \ell(y, x,\theta) &= \psi (x^T\theta)-y x^T \theta\\
    \nabla_{\theta}\ell(y,x,\theta)&=x(\dot{\psi}(x^T\theta)-y)\\
    \nabla^2_{\theta}\ell(y,x,\theta)&=xx^T\ddot{\psi}(x^T\theta)
\end{align*}
First, observe that $\theta_0$  is the unique minimizing value of $G(\theta)=\mathbb{E}[\ell(y,x,\theta)]$. This is because $\nabla_{\theta}\mathbb{E}[\ell(y,x,\theta)] = \mathbb{E}[x(\dot{\psi}(x^T\theta)-y)]=\mathbb{E}[x\mathbb{E}[(\dot{\psi}(x^T\theta)-y)\vert x]]=0$ at $\theta_0$. Therefore $\theta_0$ is a local minimum of $G(\theta)$, and since $\psi$ is strictly convex so is $G(\theta)$, hence $\theta_0$ is the unique global minimum. Moreover, by strict convexity and positive definiteness of $\mathbb{E}[xx^T]$, the Hessian $C=\nabla^2_{\theta}G(\theta)=\mathbb{E}[xx^T\ddot{\psi}(x^T\theta)]$ is positive definite. The covariance equals 
\begin{align*}
    C_{\triangle}&=Cov(x(\dot{\psi}(x^T\theta_0)-y))\\
    &=\mathbb{E}[xx^T(\dot{\psi}(x^T\theta_0)-y)^2]\\
    &=\mathbb{E}[xx^T\mathbb{E}[(\dot{\psi}(x^T\theta_0)-y)^2\vert x]]\\
    &=\tau^{-1}\mathbb{E}[xx^T\ddot{\psi}(x^T\theta_0)]=\tau^{-1}C
\end{align*}
where we have used the fact that $\mathbb{E}[(\dot{\psi}(x^T\theta_0)-y)^2\vert x]=Var(y\vert x)=\tau^{-1}\ddot{\psi}(x^T\theta_0)$ for exponential family distributions. Further, the assumption $\mathbb{E}[\Vert xx^Tsup_{\theta\in U}\ddot{\psi}(x^T\theta)\Vert^2]<\infty$ implies %condition $R2$ in (\ref{regularity conditions}), hence
by Remark~\ref{remark sufficient Lipschitz condition} the SLD condition \eqref{stochastic Lipschitz differentiability condition}. %and %for $\theta_1,\theta_2\in U$, $|\nabla^2_{\theta}G(\theta_1)-\nabla^2_{\theta}G(\theta_2)|= \mathbb{E}[xx^T\ddot\psi(x^T\theta_1)-]$ %is satisfied 
Lastly, we obtain uniform bound over compact sets; $\sup_{\theta\in K}|\ell(y,x,\theta)|\leq \sup_{\theta\in K}|\psi(x^T\theta)|+|y|\|x\|\sup_{\theta\in K}\|\theta\|$, which has finite expectation by assumption. Hence all conditions in Theorem~\ref{main pattern robust theorem} are satisfied and the result follows.
 
\end{proof}

\begin{example}
The above theory applies to Linear, Logistic, or Poisson regression:    

\begin{table}[h!]
\centering
\begin{tabular}{|c|c|c|c|c|}
\hline
\textbf{Regression} & $\psi(\eta)$ & $\ell(y,x,\theta)$ & $C$ & $C_{\triangle}$ \\
\hline
Linear $y\sim\mathcal{N}(\eta, \sigma^2)$   & $\eta^2/2$ & $x^T\theta^2/2-yx^T\theta$ & $\mathbb{E}[xx^T]$ & $\sigma^2C$ \\
\hline
Logistic $y\sim Ber(e^\eta/(1+e^\eta))$ & $log(1+e^{\eta})$ & $log(1+e^{x^T\theta})- yx^T\theta$ & $\mathbb{E}[xx^Te^{x^T\theta_0}/(1+e^{x^T\theta_0})^2]$ & $C$ \\
\hline
Poisson $y\sim Pois(e^\eta))$   & $e^\eta$ & $e^{x^T\theta} - yx^T\theta$ & $\mathbb{E}[xx^Te^{x^T\theta_0}]$ & $C$ \\
\hline
\end{tabular}
\caption{Linear, Logistic and Poisson regression, where $\eta=x^T\theta$.}
\end{table}
\end{example}

\subsection{Robust regression}
We assume a linear relationship $y = x^T\theta_0 + \varepsilon$, with $\varepsilon$ independent of $x$, but now drop the assumption on the existence of $\mathbb{E}[\varepsilon]$, necessary to model the conditional mean $\mathbb{E}[y|x]$. We consider robust regression such as the Huber regression and Quantile regression. In both cases, the goal is to verify that conditions of Theorem~\ref{main pattern robust theorem} are satisfied, and therefore the estimator $\hat{\theta}_n$ given by (\ref{main objective for regression}) has an error $\sqrt{n}(\hat{\theta}_n-\theta_0)$ converging in distribution and in pattern to an asymptotic limit (\ref{main objective V(u)}), specified by $C$ and $C_{\triangle}$.

\subsubsection{Huber regression}

Assume linear model $y = x^T \theta + \varepsilon$, where $\varepsilon$ is a centered random variable with bounded density $f(\varepsilon)$, which is symmetric around $0$. We assume $\varepsilon$ is independent of $x$. We do not make any assumptions on the existence of $\mathbb{E}[\varepsilon]$. As a result, the conditional expectation $\mathbb{E}[y|x]$ might not exist, but the conditional truncated mean $\mathbb{E}[y\mathbf{1}_{\vert y-x^T\theta_0\vert\leq k}/\mathbb{P}[\vert y-x^T\theta_0\vert\leq k]\vert x]=x^T\theta_0$ does. We assume that $\mathbb{E}[xx^T]$ is positive definite, and that $\mathbb{E}[\Vert x\Vert^4]<\infty$.

Consider $ \ell(y, x, \theta)= H(y-x^T\theta)= H(\varepsilon - t(\theta))$, with $t(\theta) := x^T(\theta-\theta_0)$ and
the Huber loss $H:\mathbb{R}\rightarrow\mathbb{R}$, given by
\begin{align*}
    H(x)&:=k \left(-x-\dfrac{k}{2}\right)\mathbf{1}_{\{x<-k\}}+\dfrac{1}{2}x^2\mathbf{1}_{\{\vert x \vert<k\}}+k\left(x-\dfrac{k}{2}\right)\mathbf{1}_{\{x>k\}},
\end{align*}
for some constant $k>0$.
\begin{proposition}
    Under the above assumptions, $\hat{u}_n=\sqrt{n}(\hat{\theta}_n-\theta_0)$ converges in distribution and in pattern to $\hat{u}$, which minimizes \eqref{main objective V(u)}, with $C=\delta_k\mathbb{E}[xx^T], C_{\triangle}=\gamma_k\mathbb{E}[xx^T]$,
    \begin{align*}
        \delta_k&=\mathbb{E}[\varepsilon^2\mathbf{1}_{\vert\varepsilon\vert<k}]+k^2\mathbb{P}[\vert\varepsilon\vert>k],\\
        \gamma_k&=\mathbb{P}[\vert\varepsilon\vert<k].
    \end{align*}
\end{proposition} 
\begin{proof}

To show this, we verify the conditions of Theorem~\ref{main pattern robust theorem}.
To that end, first consider
\begin{align*}
    D(\varepsilon,t):=\dfrac{\partial}{\partial t}H(\varepsilon-t)&= k\mathbf{1}_{\{\varepsilon-t<-k\}}-(\varepsilon-t)\mathbf{1}_{\{\vert \varepsilon-t\vert \leq k\}}-k\mathbf{1}_{\{\varepsilon-t>k\}}.%\\
    %\dfrac{\partial^2}{\partial^2 t} H(x-t)&= \mathbf{1}_{\{\vert x-t\vert \leq k\}}
\end{align*}
By chain rule $\nabla_{\theta}\hspace{0,1cm}\ell(y,x,\theta)=x D(\varepsilon,t(\theta))$, hence $\triangle(x,\theta_0)=\nabla\vert_{\theta=\theta_0}\hspace{0,1cm}\ell(y,x,\theta)=xD(\varepsilon,0)$ and we have $\mathbb{E}[\triangle(x,\theta_0)]=\mathbb{E}[x]\mathbb{E}[D(\varepsilon,0)]=0$, because $\mathbb{E}[D(\varepsilon,0)]=0$ by the symmetry of $\varepsilon$. We now briefly verify SLD at $\theta_0$ (\ref{stochastic Lipschitz differentiability condition}). By remark \ref{remark sufficient Lipschitz condition}, it suffices to argue the Lipschitzness of the derivative $\theta\mapsto\nabla_{\theta}\hspace{0,1cm}\ell(y,x,\theta)$; 
\begin{align*}
\Vert x D(\varepsilon,t(\theta_1))-x D(\varepsilon,t(\theta_2))\Vert&=\Vert x\Vert\cdot \vert D(\varepsilon,t(\theta_1))- D(\varepsilon,t(\theta_2))\vert\\
&\leq \Vert x\Vert\cdot \vert t(\theta_1)- t(\theta_2)\vert\\
&\leq \Vert x\Vert^2\cdot \Vert \theta_1-\theta_2\Vert,
\end{align*}
and using that the components of $x$ have fourth moments, gives (\ref{stochastic Lipschitz differentiability condition}), and hence also local stochastic differentiability $i)$.
Further, from expansion (\ref{Taylor expansion of loss}) we consider the term $\triangle(\cdot, \theta_0)=x D(\varepsilon,0)$, which has the covariance matrix $C_{\triangle}=\delta_k \mathbb{E}[xx^T]$, where 
\begin{equation*}
    \delta_k=\mathbb{E}[D(\varepsilon,0)^2]=\mathbb{E}[\varepsilon^2\mathbf{1}_{\vert\varepsilon\vert<k}]+k^2\mathbb{P}[\vert\varepsilon\vert>k].
\end{equation*}
Moreover, taking the second derivative in the weak sense $\nabla^2_{\theta}\hspace{0,1cm}\ell(y,x,\theta)=xx^T\mathbf{1}_{\{\vert \varepsilon-t(\theta)\vert \leq k\}}$, we obtain $C=\nabla^2\vert_{\theta=\theta_0}G(\theta)=\gamma_k C$, where $\gamma_k=\mathbb{P}[\vert\varepsilon\vert<k]$. Because $\mathbb{E}[xx^T]$ exists and is positive definite and $\delta_k,\gamma_k<\infty$, conditions $ii),iii)$ are verified. Finally, $\hat{\theta}_n$ is consistent hence also tight, so $iv)$ holds. We do not need to check condition $v)$, because this only serves to establish consistency. The last remaining condition to verify is Lipschitzness of $\theta\mapsto \nabla_{\theta}^2G(\theta)=\mathbb{E}[xx^T\mathbf{1}_{\{\vert \varepsilon-t(\theta)\vert \leq k\}}]=\mathbb{E}[xx^T\mathbb{E}[\mathbf{1}_{\{\vert \varepsilon-t(\theta)\vert \leq k\}}\vert x]]$ at $\theta_0$. For $\theta_1,\theta_2$ we get
\begin{align*}
    \| \nabla_{\theta}^2G(\theta_1)-\nabla_{\theta}^2G(\theta_2)\|&=\left\|\mathbb{E}\left[xx^T \left(\int_{t(\theta_1)-k}^{t(\theta_1)+k}f(\varepsilon)d\varepsilon-\int_{t(\theta_2)-k}^{t(\theta_2)+k}f(\varepsilon)d\varepsilon\right) \right]\right\|\\
    &\leq \left\|\mathbb{E}\left[xx^T \max_{\varepsilon\in\mathbb{R}} f(\varepsilon)\vert t(\theta_1)-t(\theta_2)\vert \right]\right\|\\
    &\leq \max_{\varepsilon\in\mathbb{R}} f(\varepsilon) \mathbb{E}\left[\|xx^T\|\cdot\Vert x\Vert\right]\cdot \Vert \theta_1-\theta_2\Vert,
\end{align*}
where the constant in the last expression is finite by boundedness of $f$ and finiteness of third moments of $x$. We have verified all conditions in Theorem~\ref{main pattern robust theorem}, which finishes the proof.
\end{proof}
%\end{example}
Note that letting $k\rightarrow\infty$, we get $\delta_k\rightarrow\mathbb{E}[\varepsilon^2]=\sigma^2$ and $\gamma_k\rightarrow 1$, converging to the standard quadratic loss with $C=\mathbb{E}[xx^T]$ and $C_{\triangle}=\sigma^2 C$.

\subsubsection{Quantile regression}
Quantile regression \cite{koenker1978regression} is a popular method where one assumes that there is a linear relation between the fixed effect and a fixed  quantile of $y$. It has been studied earlier with $L_1$ regularization in \cite{l1normquantile}.

For the Quantile regression we drop the assumption of finite variance and zero mean of the noise $\varepsilon$. Instead
we assume that $\varepsilon$ has a continuous density $f$ and CDF $F$, which satisfies $F^{-1}(\alpha)=\mathbb{P}(\varepsilon\leq 0)=0$, and $f(0)>0$. This assumption implies that there is a linear relationship between the covariate $x$ and the conditional $\alpha$-quantile $$Q_{\alpha}(y|x)=\inf\{t\in\mathbb{R}: F_{y|x}(t)\leq \alpha\}=x^T\theta_0+Q_{\alpha}(\varepsilon|x)=x^T\theta_0.$$

Finally, assume that the density $f$ is bounded and Lipschitz continuous at $0$. Also, we assume that $\mathbb{E}[xx^T]$ is positive definite, and $\mathbb{E}[\Vert x\Vert^3]<\infty$. Consider $
    \ell(y, x, \theta)= \vert y-x^T\theta\vert_{\alpha} =\vert\varepsilon - t(\theta)\vert_{\alpha}$, with $t(\theta) := x^T(\theta-\theta_0)$, 
    where $\vert \cdot \vert_{\alpha} :\mathbb{R}\rightarrow\mathbb{R}$, given by
\begin{align*}
    \vert x \vert_{\alpha}&:=\begin{cases}
      (1-\alpha)(-x) &  x\leq0, \\
      \alpha x &  x>0,
    \end{cases}
\end{align*}
and $\alpha\in (0,1)$. Note that $\vert x \vert_{1/2}=\frac{1}{2}\vert x\vert$.
\begin{proposition}
    Under the above assumptions, $\hat{\theta}_n$ given by \eqref{main objective for regression}, has an error $\hat{u}_n=\sqrt{n}(\hat{\theta}_n-\theta_0)$, which converges in distribution and in pattern to $\hat{u}$, minimizing \eqref{main objective V(u)}, where
\begin{align*}
    C&=\alpha(1-\alpha)\mathbb{E}[xx^T],\\
    C_{\triangle}&=f(0)\mathbb{E}[xx^T].
\end{align*}
\end{proposition}
\begin{proof}
As for the Huber regression, we verify conditions of Theorem~\ref{main pattern robust theorem}. We verify the SLD condition (\ref{stochastic Lipschitz differentiability condition}) for the quantile loss in appendix \ref{appendix stochastic Lipschitz differentiability for the Huber loss}. In particular, this gives $i)$. Furthermore, the loss function $\vert \cdot \vert_{\alpha}$ is everywhere differentiable except at $0$, with weak derivative $\vert x \vert_{\alpha}^{'}= - (1-\alpha)\mathbf{1}_{\{x \leq 0\}}+\alpha \mathbf{1}_{\{x > 0\}}$. Note that $t\mapsto\mathbb{E}[\vert\varepsilon-t\vert_{\alpha}]$ attains its minimum at the solution to $0=- (1-\alpha)\mathbb{P}(\varepsilon \leq t)+\alpha \mathbb{P}(\varepsilon > t)$, which is the $\alpha$ quantile of $\varepsilon$. This minimum is obtained by setting $t=0$, because $\mathbb{P}(\varepsilon\leq 0)=\alpha$. Since $t(\theta_0)=0$, it follows that $G(\theta)$ has a minimum value at $\theta_0$.

By chain rule $\nabla_{\theta}\hspace{0,1cm}\ell(y,x,\theta) = -x \vert \varepsilon - t(\theta) \vert_{\alpha}^{'}=x((1-\alpha)\mathbf{1}_{\{\varepsilon \leq t(\theta)\}}-\alpha \mathbf{1}_{\{\varepsilon > t(\theta)\}})$ and the $\triangle(\cdot,\theta_0)$ term in expansion (\ref{Taylor expansion of loss}) reads $\triangle(\varepsilon,x,\theta_0)=-x \vert \varepsilon\vert_{\alpha}^{'}$. Consequently, the covariance matrix of $\triangle(\cdot,\theta_0)$ equals $C_{\triangle}=\delta \mathbb{E}[xx^T]$, where
\begin{align*}
    \delta= \mathbb{E}\bigr[ (\vert \varepsilon_i\vert_{\alpha}^{'})^2\bigr]&=  (1-\alpha)^2\mathbb{P}(\varepsilon_i\leq 0)+\alpha^2\mathbb{P}(\varepsilon_i> 0)\\
    &=(1-\alpha)^2\alpha+\alpha^2(1-\alpha)\\
    &=\alpha(1-\alpha).
\end{align*}
This verifies condition $iii)$. Moreover,
\begin{align*}
    \nabla_{\theta}\mathbb{E}\bigr[\hspace{0,1cm}\ell(y,x,\theta)\bigr]&=\mathbb{E}\bigr[\hspace{0,1cm}\nabla_{\theta}\ell(y,x,\theta)\bigr]\\
    &=\mathbb{E}\Bigr[x\mathbb{E}\bigr[(1-\alpha)\mathbf{1}_{\{\varepsilon \leq t(\theta)\}}-\alpha \mathbf{1}_{\{\varepsilon_i > t(\theta)\}}\vert x\bigr]\Bigr]\\
    &=\mathbb{E}\Bigr[x\bigr((1-\alpha)\mathbb{P}({\varepsilon \leq t(\theta)} \vert x)-\alpha(1-\mathbb{P}({\varepsilon \leq t(\theta)} \vert x))\bigr)\Bigr]\\
    &= \mathbb{E}\Bigr[x\bigr(\mathbb{P}({\varepsilon \leq t(\theta)} \vert x)-\alpha\bigr)\Bigr],
\end{align*}

Thus $\nabla^2_{\theta}G(\theta)=\mathbb{E}[f(t(\theta))xx^T]$, which exists for every $\theta$ by boundedness of $f$, and it is Lipschitz at $\theta_0$ by Lipschitzness of $f$ at $0$. Indeed, for $\theta_1,\theta_2$ around $\theta_0$, we can bound
$$\vert f(t(\theta_1))-f(t(\theta_2))\vert\leq L \vert x^T(\theta_1-\theta_2)\vert\leq L \Vert x\Vert \cdot \Vert \theta_1 - \theta_2\Vert.$$
Therefore
\begin{equation*}
    \| \nabla^2_{\theta}G(\theta_1)-\nabla^2_{\theta}G(\theta_2)\|=\|\mathbb{E}[f(t(\theta_1))-f(t(\theta_2)) xx^T]\|\leq L\cdot \mathbb{E}\left[\Vert x \Vert \cdot\|xx^T\|\right] \Vert\theta_1-\theta_2\Vert.
\end{equation*}
The expectation in the last term exists because $\mathbb{E}[\Vert x\Vert^3]<\infty$. Finally, we obtain the Hessian matrix $C=\nabla^2\vert_{\theta=\theta_0}G(\theta)=f(0) \mathbb{E}[xx^T]$, and $ii)$ holds. Also, $iv)$ holds since $\hat{\theta}_n$ is consistent, hence also tight. Consequently, Theorem~\ref{main pattern robust theorem} applies. 
\end{proof}

\begin{example}
Let $\alpha=1/2$, so that $\ell(y,x,\theta)=\vert y-x^T\theta_0\vert/2$ is the median loss. Assume that $\varepsilon$ follows a Laplace distribution with density $f(\varepsilon)=e^{-\vert \varepsilon\vert}/2$. Setting $t=t(\theta)=x^T(\theta-\theta_0)$, we can write $G(\theta)=\mathbb{E}[(1/2)\int\vert\varepsilon -t\vert f(\varepsilon)d\varepsilon|x]$. Focusing on the inner integral $\tilde{G}(t)=(1/2)\int\vert\varepsilon -t\vert f(\varepsilon)d\varepsilon$ the above calculation asserts that $\tilde{G}''(t)=f(t)=e^{-\vert t\vert}/2$, which is bounded and Lipschitz continuous with Lipschitz constant $1/2$.
\end{example}

\section{Simulations}

In this section, we illustrate our asymptotic results using the example of logistic regression. We consider the case where $p = 30$ and $n \in \{500, 1000, 5000, 10000\}$. The rows of the design matrix are generated as independent and identically distributed random vectors from a multivariate normal distribution with zero mean and a block-diagonal covariance matrix satisfying
\[
\mathbb{E}[x^\top x] = I \otimes \Sigma^0,
\]
where $\Sigma^0$ is a compound symmetry $5\times5$ matrix defined by $\Sigma^0_{i,i} = 1$ and $\Sigma^0_{i,j} = 0.8$ for $i \neq j$, $I$ is the $6\times6$ identity, and $\otimes$ is the Kronecker product.

The true parameter vector $\theta^0$ is defined as follows:
\[
\theta^0_1 = \ldots = \theta^0_5 = -2, \quad \theta^0_6 = \ldots = \theta^0_{10} = 1, \quad \theta^0_{11} = \ldots = \theta^0_{30} = 0.
\]

To estimate the regression coefficients, we use the SLOPE estimator with the penalty function
\[
f(\theta) = \sum_{j=1}^p \lambda_j |\theta|_{(j)},
\]
where $|\theta|_{(1)} \geq |\theta|_{(2)} \geq \dots \geq |\theta|_{(p)}$ denote the ordered absolute values of the components of $\theta^0$, and
\[
\lambda_j = \alpha \cdot \Phi^{-1} \left(1 - \frac{0.2j}{2p} \right), \quad j = 1, 2, \dots, p,
\]
with $\Phi^{-1}(\cdot)$ denoting the quantile function (inverse CDF) of the standard normal distribution.

Figure~\ref{fig:RMSE} shows the dependence of the Root Mean Square Error,
\[
\mathrm{RMSE} = \sqrt{\mathbb{E}\left(n \|\hat{\theta}_S - \theta^0\|^2\right)},
\]
where $\hat{\theta}_S$ is the SLOPE estimator of $\theta$, as a function of the tuning parameter $\alpha$ for various sample sizes $n$. We also compare the empirical RMSE to its asymptotic counterpart, computed by averaging the squared norms of 50{,}000 independent realizations of the vector $\hat{u}_n$, obtained by optimizing the function $V(u)$ defined in equation~(\ref{main objective V(u)}), across corresponding realizations of the noise vector $W$.

As expected, for each fixed value of $\alpha$, the RMSE approaches its asymptotic value as $n$ increases. Moreover, we observe that the rate of this convergence tends to decrease as $\alpha$ increases.

Additionally, Figure~\ref{fig:RRE} illustrates the rapid convergence of the Relative Residual Error (RRE) \eqref{eq:RRE} to its corresponding asymptotic value.

\begin{figure}[ht!]
    \centering
    \begin{minipage}{0.49\linewidth}
        \centering
        \includegraphics[width=\linewidth]{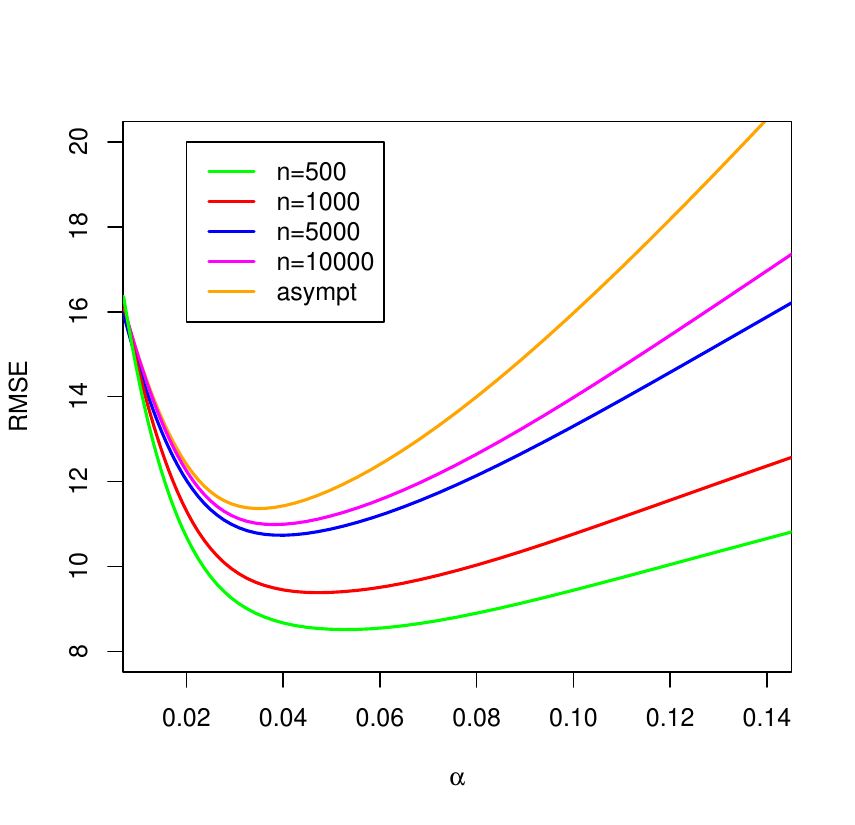}
        \caption{RMSE as a function of $n$ and $\alpha$.\\\phantom{RMSE as a function of $n$ and $\alpha$.}}
        \label{fig:RMSE}
    \end{minipage}
    \hfill
    \begin{minipage}{0.49\linewidth}
        \centering
        \includegraphics[width=\linewidth]{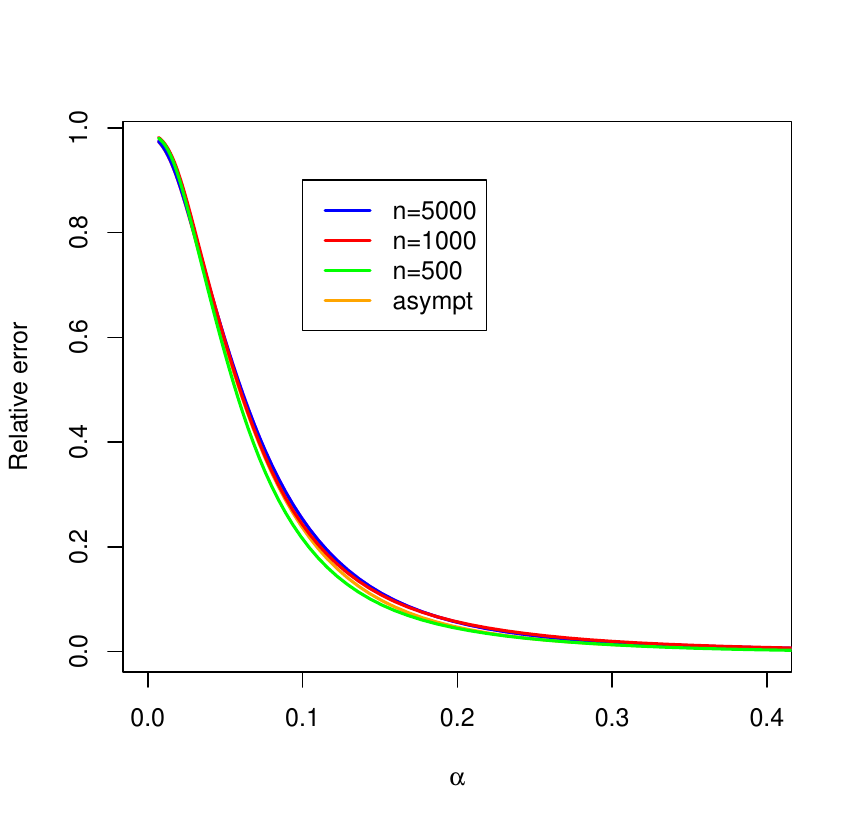}
        \caption{Relative Residual Error as a function of $n$ and $\alpha$.}
        \label{fig:RRE}
    \end{minipage}
\end{figure}

Furthermore, Figure~\ref{fig:pattern} demonstrates a systematic, though relatively slow, convergence of the probability of perfect pattern recovery to its asymptotic counterpart. This slower rate of convergence is primarily due to challenges in correctly clustering the coefficients associated with the group of smaller absolute values. Specifically, the decline in tuning parameters corresponding to this group is less pronounced, making it more difficult to separate the coefficients. In contrast, the beginning of the $\lambda$ sequence exhibits a more rapid decay, facilitating faster and more accurate clustering of the coefficients in the group with the larger absolute value.

\begin{figure}[ht!]
    \centering
    % Include your figure here
    \includegraphics[width=0.5\linewidth]{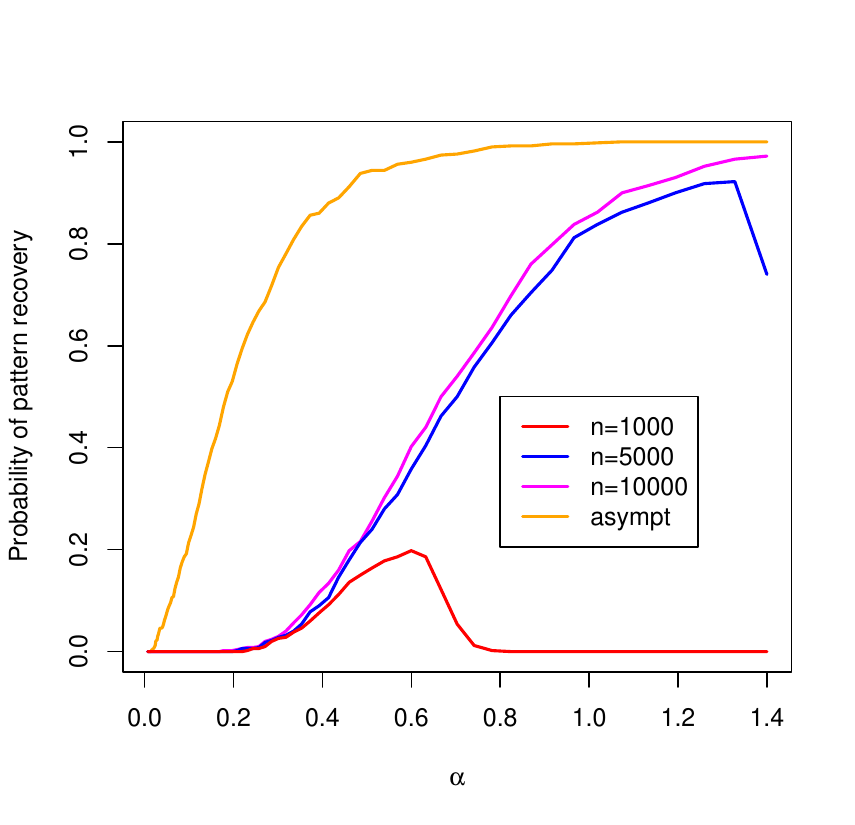}
    \caption{Probability of perfect pattern recovery as a function of $n$ and $\alpha$.}
    \label{fig:pattern}
\end{figure}

\section{Discussion}
Distributional convergence of model patterns induced by convex regularizers is a non-trivial phenomenon — an issue that has been largely overlooked in the prior literature. In this paper, through Theorem \ref{main pattern robust theorem}, we establish the asymptotic distribution of lower-dimensional model patterns associated with regularizers, including Lasso, Fused Lasso, SLOPE, and Elastic Net. Notably, building on the concept of stochastic differentiability \cite{pollard_1985}, we introduce the notion of stochastic Lipschitz differentiability. This extends the asymptotic characterization of pattern behaviour beyond linear regression to generalized linear models and accommodates robust loss functions such as the Huber and quantile losses. As a result, it unifies the asymptotic findings of \cite{pollard_1985} and \cite{fu2000asymptotics} within a broader framework.

Consequently, we characterize the limiting probability of recovering the true model using a given regularizer in the regime where $p$ is fixed and $n$ approaches infinity. Despite its limitations, we believe that our framework offers a new perspective on model recovery and has the potential to inspire further methodological advancements. In particular, it holds promise for second-order methods designed to recover the lower-dimensional structure within a signal.

\section*{Acknowledgments}
The authors acknowledge the support of the Swedish Research Council, grant no. 2020-05081. We would like to thank A.W. van der Vaart for his useful comments on the local stochastic differentiability condition. Further, we thank Henrik Bengtsson, Dragi Anevski, and Hermann Thorisson for their discussions and Wojciech Reichel for pointing out the application in GLMs.

\bibliographystyle{plain}
\newpage
\bibliography{citation}

\newpage

\begin{appendix}
\section{Appendix}
\subsection{Stochastic Lipschitz differentiability for the Huber loss}\label{appendix stochastic Lipschitz differentiability for the Huber loss}
We verify the statement in Remark~\ref{remark sufficient Lipschitz condition}, that if $$\Vert\ell'(x,\theta_1)-\ell'(x,\theta_2)\Vert\leq M(x) \Vert \theta_1-\theta_2\Vert$$ for $\theta_1,\theta_2\in B$ and $\mathbb{E}[M(X)^2]<\infty$, then the SLD condition (\ref{stochastic Lipschitz differentiability condition}) holds. For $\triangle(x,\theta_0):=\ell'(x,\theta_0)$,

\begin{align*}
\vert\nu_n(\tilde{r}(\cdot,\theta_1)-\tilde{r}(\cdot,\theta_2))\vert&=\vert\nu_n(\ell(\cdot,\theta_1)-\ell(\cdot,\theta_2)-\ell'(\cdot,\theta_0)(\theta_1-\theta_2))\vert\\
&=\vert\nu_n(\ell'(\cdot,\theta^*)-\ell'(\cdot,\theta_0))(\theta_1-\theta_2)\vert\\
&\leq \Vert\nu_n(\ell'(\cdot,\theta^*)-\ell'(\cdot,\theta_0))\Vert\cdot \Vert\theta_1-\theta_2\Vert,%\\
%&=o_p(1)\Vert\theta_1-\theta_2\Vert,
\end{align*}
for some $\theta^*$ between $\theta_1$ and $\theta_2$. Now, the family $\mathcal{F}=\{\ell'(\cdot,\theta):\theta\in B\}$  is a Donsker class (see Example 19.7 \cite{van1998asymptotic}). This means that $\nu_n\ell'(\cdot,\theta)$ converges weakly as a stochastic process in the space $\ell^{\infty}(B)$ with the norm $\Vert f\Vert_B:=\sup_{\theta\in B}\Vert f(\theta)\Vert$. Thus for any $\delta>0$, by the Portmanteau Lemma:
\begin{align*}
    \limsup_{n\rightarrow\infty}\mathbb{P}[\Vert\nu_n(\ell'(\cdot,\theta)-\ell'(\cdot,\theta_0))\Vert_B\geq \delta]\leq \mathbb{P}[\Vert Z(\theta)\Vert_B\geq \delta]\leq \delta^{-1}\mathbb{E}[\Vert Z(\theta)\Vert_B],
\end{align*}
where $Z$ is the corresponding limiting centered Gaussian process with $Z(\theta_0)=0$. %The latter probability goes to zero if we let the ball $B$ shrink toward $\theta_0$.
By the Dudley's integral inequality (see for instance \cite{vershynin2018high}) one can bound the supremum of a centered Gaussian process:
\begin{equation*}
    \mathbb{E}[\Vert Z(\theta)\Vert_B]\leq c \int_0^{\operatorname{diam}(B)}\sqrt{\log \hspace{0,1cm}N(B,d,\varepsilon)}d\varepsilon,
\end{equation*}
where $N(B,d,\varepsilon)$ is the number of closed $\varepsilon$-balls with centers in $B$, needed to cover $B$, with distance defined as $d(\theta_1,\theta_2)=\mathbb{E}[\Vert Z(\theta_1)-Z(\theta_2)\Vert^2]^{1/2}$, and $c$ a constant not depending on $B$. The diameter is defined as $\operatorname{diam}(B)=\sup\{\|x-y\|:x,y\in B\}$. Now crucially, by the Lipschitzness assumption on $\ell'(x,\theta)$;
\begin{equation*}
    d(\theta_1,\theta_2)=\mathbb{E}[\Vert Z(\theta_1)-Z(\theta_2)\Vert^2]^{1/2}=\mathbb{E}[\Vert\ell'(X,\theta_1)-\ell'(X,\theta_2)\Vert^2]^{1/2}\leq\gamma\Vert\theta_1-\theta_2\Vert,
\end{equation*}
where $\gamma:=\mathbb{E}[M(X)^2]^{1/2}$, and $\Vert\cdot\Vert$ denotes the usual Euclidean distance. Therefore an $\varepsilon$-ball w.r.t. $\gamma\Vert\cdot\Vert$ is smaller than $\varepsilon$-ball w.r.t. $d$, i.e. $\{\theta:\gamma\|\theta\|<\varepsilon\}\subset \{\theta:d(\theta,0)<\varepsilon\}$, and thus $N(B,d,\varepsilon)\leq N(B,\gamma\Vert\cdot\Vert,\varepsilon)$. %\leq %c_2 (\gamma diam(B)/\varepsilon)^p$ 
The number of Euclidean $\varepsilon$-balls needed to cover $B\subset\mathbb{R}^p$ is proportional to $(\operatorname{diam}(B)/\varepsilon)^p$, and we can bound the above integral from above to obtain
\begin{equation*}
\mathbb{E}[\Vert Z(\theta)\Vert_B]\leq \tilde{c}\cdot \operatorname{diam}(B)\int_0^{1}\sqrt{\log \hspace{0,1cm}((1/\varepsilon)^p)}d\varepsilon.
\end{equation*}
The displayed integral exists, and $\mathbb{E}[\Vert Z(\theta)\Vert_B]$ converges to $0$ as $\operatorname{diam}(B)\downarrow 0$.

Consequently, for %$\theta^*$  
$\theta_1,\theta_2\in\theta_0+K/\sqrt{n}$, 
$\Vert\nu_n(\ell'(\cdot,\theta^*)-\ell'(\cdot,\theta_0))\Vert=o_p(1)$, which proves (\ref{stochastic Lipschitz differentiability condition}).

\subsection{Stochastic Lipschitz condition for the Quantile loss}\label{appendix stochastic Lipschitz for Quantile}
We verify SLD (\ref{stochastic Lipschitz differentiability condition}) for median loss $\ell(y,x,\theta)=\vert y-x^T\theta\vert/2$ at the point $\theta_0$. For quantile, the argument is analogous. We only assume $\mathbb{E}[\Vert x\Vert^2]<\infty$. Consider the family:
\begin{equation*}
   \mathcal{F}_B = \left\{M_{\theta_1,\theta_2}(y,x):=\dfrac{\tilde{r}(y,x,\theta_1)-\tilde{r}(y,x,\theta_2) }{\Vert\theta_1-\theta_2\Vert}: \theta_1,\theta_2\in B, \theta_1\neq\theta_2\right\},
\end{equation*}
for some ball $B$ around $\theta_0$ and $\tilde{r}(\cdot,\theta_1)=\ell(\cdot,\theta)-\ell(\cdot,\theta_0)-(\theta-\theta_0)^T\triangle(\cdot,\theta_0)$. We have
\begin{align*}
    M_{\theta_1,\theta_2}(y,x)&=\dfrac{\vert y-x^T\theta_1\vert-\vert y-x^T\theta_2\vert-(\theta_1-\theta_2)^Tx\cdot \operatorname{sgn}(y-x^T\theta_0)}{2\Vert\theta_2-\theta_1\Vert}.
\end{align*}

One can verify that
\begin{equation*}
    \vert M_{\theta_1,\theta_2}(y,x) \vert \leq \Vert x\Vert \mathbf{1}_{\{\vert y-x^T\theta_0\vert\leq\max\{\vert x^T(\theta_1-\theta_0)\vert, \vert x^T(\theta_2-\theta_0)\vert\}\}}.
\end{equation*}
In particular, $M_{\theta_1,\theta_2}(y,x)$ vanishes with high probability for $\theta_1,\theta_2$ close to $\theta_0$. Also, the family $\mathcal{F}_B$ has an envelope function $F(y,x)=\Vert x\Vert$. We now show that $\mathcal{F}_B$ has a finite VC-dimension. 
Reparametrizing with $z=(y,x)$, and $v_i=(1,-\theta_i) \in\mathbb{R}^{p+1}$, for $i=0,1,2$; %$v_0=(1,-\theta_0)$, $v_1=(1,-\theta_1), v_2=(1,-\theta_2) \in\mathbb{R}^{p+1}$, gives
\begin{align*}
    M_{v_1,v_2}(z)=\dfrac{\vert v_1^Tz\vert-\vert v_2^T z\vert + (v_2-v_1)^Tz\cdot \operatorname{sgn}(v_0^Tz)}{2\Vert v_1-v_2\Vert}.
\end{align*}

Define the linear functions $f_1(z)=v_1^Tz/(2\Vert v_1-v_2\Vert), f_2(z)=v_2^Tz/(2\Vert v_1-v_2\Vert)$. For simplicity, consider $M_{v_1,v_2}(z)$ only in the region $v_0^Tz>0$. %$\varepsilon>0$. 
(The argument for $v_0^Tz<0$ %$\varepsilon<0$ 
is analogous.) Since for any $a,b,c,d\in\mathbb{R}$:
$$a\lor b- c\lor d=((a-c)\land(a-d))\lor((b-c)\land (b-d)),$$
where $a\lor b:=\max\{a,b\}, a\land b:= \min\{a,b\}$, we get
\begin{align*}
    M_{v_1,v_2}&=(f_1\lor-f_1)-(f_2\lor-f_2) + (f_2-f_1)\\
    &=(((f_1-f_2)\land (f_1+f_2))\lor((-f_1-f_2)\land (-f_1+f_2))) + (f_2-f_1)\\
    &=(0\land 2(f_1+f_2))\lor(-2f_1\land(-2f_1+2f_2)),
\end{align*}
 Consequently, the family $\mathcal{F}_B$ %$=\{M_{v_1,v_2}: v_1, v_2 \in \{1\}\times\mathbb{R}^p, v_1\neq v_2\}$ 
 has finite VC-dimension $V(\mathcal{F}_B)$, as it can be represented as $\mathcal{F}=(0 \land \mathcal{F}_1)\lor (\mathcal{F}_2\land\mathcal{F}_3)$, with $\mathcal{F}_1,\mathcal{F}_2,\mathcal{F}_3$ consisting only of linear functions. As a result, $\mathcal{F}_B$ is a Donsker class and the empirical process $\nu_n M_{\theta_1,\theta_2}(\cdot)$ converges weakly in the normed space $\ell^{\infty}(\mathcal{F}_B)\cong \ell^{\infty}(B^2)$, where $B^2:=\{(\theta_1,\theta_2)\in B\times B, \theta_1\neq\theta_2\}$ and $\Vert f(\theta_1,\theta_2)\Vert_{B^2}:=\sup_{(\theta_1,\theta_2)\in B^2}\vert f(\theta_1,\theta_2)\vert$, 
 to a centered Gaussian process $Z(\theta_1,\theta_2)$. By Theorem 2.6.7 in \cite{vanderVaart1996}, the covering numbers are polynomial in $V(\mathcal{F}_B)$:
 \begin{equation*}
%N(\varepsilon \mathbb{E}[\Vert X \Vert^2]^{1/2}, B, \rho)=
N(\varepsilon \Vert F\Vert_{L^2(P)}, \mathcal{F}_B, L_2(P))\leq K \varepsilon^{-2(V(\mathcal{F}_B)-1)},
 \end{equation*}
where $P$ is the law of $(Y,X)$, and $\Vert F\Vert_{L^2(P)}^2= \int F(y,x)^2dP = \mathbb{E}[F(X,Y)^2] = \mathbb{E}\Vert X\Vert^2<\infty$ %$\rho((\theta_1,\theta_2)=(\theta_3,\theta_4))=\mathbb{E}[(M_{\theta_1,\theta_2})]$. 
By the Dudley's inequality (Corollary 2.2.8 \cite{vanderVaart1996}),
\begin{align*}
    %\mathbb{E}[sup_{\theta_1\neq\theta_2\in B}\vert Z(\theta_1,\theta_2)\vert]
    \mathbb{E}[\Vert Z(\theta_1,\theta_2)\Vert_{B^2}] \leq K \int_{0}^{\operatorname{diam}(B^2)}\sqrt{\log N(\varepsilon, \mathcal{F}_B, L_2(P))}d\varepsilon,
\end{align*}
which converges by the above VC-bound and DCT to $0$, as $\operatorname{diam}(B)\rightarrow 0$. Finally, using that $\mathcal{F}_B$ is $P-$ Donsker, we obtain: 
\begin{align*}
    \limsup_{n\rightarrow\infty}\mathbb{P}[\Vert\nu_n(M_{\theta_1,\theta_2})\Vert_{B^2}\geq \delta]\leq \mathbb{P}[\Vert Z(\theta_1,\theta_2)\Vert_{B^2}\geq \delta]\leq \delta^{-1}\mathbb{E}[\Vert Z(\theta_1,\theta_2)\Vert_{B^2}],
\end{align*}
which converges to $0$ as $\operatorname{diam}(B)\rightarrow 0$ and (\ref{stochastic Lipschitz differentiability condition}) follows.

\subsection{Lipschitzness of the Taylor rest term}\label{appendix Lipschitzness of the Taylor rest term}
Let $f:\mathbb{R}^p\rightarrow \mathbb{R}$ be a $C^2$ function in some neighborhood $U$ around $x_0$, with Taylor expansion
\begin{equation*}
    f(x)=\underbrace{f(x_0)+f'(x_0)(x-x_0)+(x-x_0)^Tf''(x_0)(x-x_0)/2}_{P(x)} + r(x)\Vert x-x_0\Vert^2,
\end{equation*}
where $r(x)\rightarrow 0 $ as $x\rightarrow x_0$. If $f''(x)$ is locally Lipschitz continuous at the point $x_0$, i.e. there exists a constant $L_f\geq 0$ such that
\begin{equation*}
    \Vert f''(x)-f''(x_0) \Vert \leq L_f \Vert x -x_0\Vert\;\;\forall x \in U,
\end{equation*}
then $r(x)$ is Lipschitz continuous on $U$.
\begin{proof}
    Let wlog $x_0=0$. Observe that the function $g(x):=f(x)-P(x)$ has vanishing first two derivatives at zero; $g(0)=0, g'(0)=0, g''(0)=0$. Furthermore, $g''$ is Lipschitz continuous at zero, i.e. there exists a constant $L\geq 0$ such that;
    \begin{equation}\label{lipschitzness of g'' around 0}
        \Vert g''(x)\Vert = \Vert g''(x)-g''(0) \Vert \leq L \Vert x \Vert,
    \end{equation}
     for any $x\in U$. We have $r(x)=g(x)/\Vert x\Vert^2$. For $x,y\in U$ we want to bound
    \begin{equation*}
        \frac{\vert r(x) -r(y)\vert}{\Vert x-y\Vert}=\frac{h(x,y)}{\Vert x-y\Vert\cdot \Vert x\Vert^2\cdot \Vert y\Vert^2},
    \end{equation*}
where $$h(x,y):=\vert\Vert y\Vert^2 g(x)-\Vert x\Vert^2 g(y)\vert. $$
%We express $g(x)=x^Tg''(\zeta)x/2$, for some $\zeta$ on the line between $0$ and $x$. 
We shall wlog assume that $\Vert y\Vert\geq \Vert x \Vert$. Since $g$ is $C^2$ on $U$, we can expand $g(y)$ around $x$ as
\begin{align*}
    g(y)&=g(x)+g'(x)(y-x)+(y-x)^Tg''(\xi)(y-x)/2,%\underbrace{g(x)}_{x^Tg''(\zeta)x/2}+\underbrace{g'(x)}_{\chi^T g''(\chi)}(y-x)+(y-x)^Tg''(\xi)(y-x)/2,%\\
    %&=x^Tg''(\zeta)x/2+(y-x)^Tg''(\xi)(y-x)/2+g''(\chi)(y-\chi)(y-x),
\end{align*}
for some $\xi$ on the line connecting $x$ and $y$.\footnote{This follows from second order Taylor expansion of the $C^2$ function $\tilde{g}(t)=(g \circ \gamma)(t):\mathbb{R}\mapsto\mathbb{R}$, where $\gamma:\mathbb{R}\mapsto\mathbb{R}^p$, $\gamma(t)=x+t(y-x)$. Then $g(y)-g(x)$ equals $\tilde{g}(1)-\tilde{g}(0)=\tilde{g}'(0)(1-0)+(1/2)\tilde{g}''(t^*)(1-0)^2$ for some $t^*\in[0,1]$. Finally, by chain rule $\tilde{g}'(0)=g'(x)(y-x)$ and $\tilde{g}''(t^*)=(y-x)^Tg''(\xi)(y-x)$ for $\xi=\gamma(t^*)$. } This gives%., and $\chi$ between $0$ and $x$, which arises from expanding $g'(x)$ around $0$.
\begin{align*}
    h(x,y) &\leq \underbrace{(\Vert y\Vert^2-\Vert x\Vert^2) \vert g(x) \vert}_{T_1} + \underbrace{\Vert x\Vert^2\vert g'(x)(y-x)\vert}_{T_2}+\underbrace{\Vert x\Vert^2\vert(y-x)^Tg''(\xi)(y-x)/2 \vert}_{T_3}
\end{align*}
We expand $g$ around zero as $g(x)=x^Tg''(\zeta)x/2$, for some $\zeta$ on the line segment between $0$ and $x$. Using (\ref{lipschitzness of g'' around 0}) and $\Vert \zeta\Vert\leq \Vert y\Vert$, we bound $$\vert g(x)\vert=\vert x^Tg''(\zeta)x/2\vert\leq \Vert x\Vert^2 \Vert g''(\zeta)\Vert/2\leq \Vert x \Vert^2  L\Vert\zeta\Vert /2 \leq \Vert x \Vert^2  L\Vert y\Vert /2.$$ Finally, by the triangle inequality $$\Vert y\Vert^2-\Vert x\Vert^2 =( \Vert y\Vert-\Vert x\Vert ) \cdot (\Vert x\Vert + \Vert y \Vert)\leq \Vert y-x\Vert 2\Vert y\Vert,$$ so that
\begin{align*}
    T_1 %&= ( \Vert y\Vert-\Vert x\Vert ) \cdot (\Vert x\Vert + \Vert y \Vert)\vert x^Tg''(\zeta)x/2 \vert\\
    %&\leq \Vert y-x\Vert (2\Vert y \Vert) \cdot \Vert x \Vert^2 \Vert g''(\zeta)\Vert/2\\
    %&\leq \Vert y-x\Vert \cdot \Vert y \Vert \cdot \Vert x \Vert^2 \Vert \zeta\Vert L\\
    &\leq \Vert y-x\Vert \cdot \Vert y \Vert^2  \Vert x \Vert^2  L,
\end{align*}
%where we have used Lipschitzness of $g''$ at zero to bound $\Vert g''(\zeta)\Vert\leq L\Vert \zeta\Vert$. 
Similarly, expanding $g'$ around zero yields $g'(x)=g'(0)+x^T g''(\chi)=x^T g''(\chi)$, for some $\chi$ between $0$ and $x$. Using \eqref{lipschitzness of g'' around 0} and that $\Vert x\Vert, \Vert \chi \Vert, \Vert \xi\Vert, \Vert(y-x)/2\Vert$ are all bounded by $\Vert y\Vert$, one can verify that both $T_2$ and $T_3$, are bounded by $\Vert y-x \Vert \cdot \Vert y \Vert^2 \Vert x\Vert^2 L$. Consequently, $h(x,y)\leq 3L\Vert y-x \Vert \cdot \Vert y \Vert^2 \Vert x\Vert^2 $, and the rest term satisfies $\vert r(x)-r(y)\vert\leq 3L\Vert x-y\Vert$, for any $x,y\in U$.
\end{proof}

\subsection{Weak convergence of stochastic processes}\label{appendix Measurability and convergence of stochastic processes}

Here, we recollect some definitions and results on weak convergence of stochastic processes from \cite{vanderVaart1996} Section 1.5, and Chapter 2 in \cite{Billingsley1968}. Let $K\subset\mathbb{R}^p$ be a compact set, and consider the metric space $\ell^{\infty}(K)$ of bounded functions from $K$ to $\mathbb{R}$, equipped with the distance 
$$d(z_1,z_2):=  \sup_{u\in K}\vert z_1(u)-z_2(u)\vert.$$ The metric induces a Borel $\sigma-$ field generated by the open sets in $\ell^{\infty}(K)$. A stochastic process on $K$ is an indexed collection of random variables $\{V(u): u\in K\}$ defined on the same probability space $(\Omega,\mathcal{F},\mathbb{P})$, (i.e. for every $u\in K$, $\omega\mapsto V(u)(\omega)$ is a measurable map from $\Omega$ to $\mathbb{R}$).

The space of continuous functions $C(K)$ is a separable complete subspace of $\ell^{\infty}(K)$, and if a stochastic process $V$ has continuous sample paths, that is $V(\cdot)(\omega)\in C(K)$ for every $\omega\in\Omega$, it is also a Borel measurable map $V:\Omega \rightarrow C(K)$, see for instance \cite{vanderVaart1996}. 

In our context, the loss functions $x\mapsto\ell(x,\theta)$ are measurable for every fixed $\theta$, and the observations $X_1,X_2,\dots$ are measurable maps $X_i:\Omega \rightarrow \mathbb{R}^d$. Thus $\omega\mapsto G_n(\theta)(\omega)=n^{-1}\sum_{i=1}^n \ell(X_i(\omega),\theta)$ %$ \ell(X_i(\omega),\theta)$ 
is measurable for each fixed $\theta$ and $n\in\mathbb{N}$. Therefore
\begin{align*}
 V_n(u)(\omega)&=  n\left(G_n\left(\theta_0+\dfrac{u}{\sqrt{n}}\right)(\omega)-G_n(\theta_0)(\omega)\right)+\sqrt{n}\left(f\left(\theta_0+u/\sqrt{n}\right)-f(\theta_0)\right)
\end{align*}

is a stochastic process indexed by $u\in K$. Moreover, because we assume $\theta\mapsto\ell(x,\theta)$ to be continuous for every fixed $x$, the sample paths of $V_n$ are continuous $V_n(\cdot)(\omega)\in C(K)$, and hence $V_n:\Omega \rightarrow C(K)$ is a Borel measurable map. \footnote{Continuity of paths is convenient, as working in the separable $C(K)$ space avoids measurability issues that naturally arise in the non-separable $\ell^{\infty}(K)$ space. It is known that the empirical CDF is not measurable as a map from $\Omega$ to $\ell^{\infty}(K)$, as pointed out by Chibisov (1965). Nonetheless, if $\hat{\theta}_n$ is consistent, then the conclusions of Theorem~\ref{main pattern robust theorem} do not rely on continuity of the paths of $V_n$, because the argmax Theorem~3.2.2 in \cite{van1998asymptotic} is formulated within the J.~Hoffmann-J{\o}rgensen framework of outer measures and expectations.
}

\subsection{local stochastic differentiability}\label{appendix local stochastic differentiability is weak}
Here, we show that (\ref{stochastic differentiability}) implies (\ref{local stochastic differentiability}). Let $K\subset\mathbb{R}^p$ be a compact set. Consider any sequence of balls $U_n$ shrinking to $\theta_0$ such that $\theta_0+K/\sqrt{n}\subset U_n$ for every $n$.  We observe that $\Vert K\Vert:=sup\{\Vert  u\Vert: u\in K\}\geq n^{1/2}\Vert\theta-\theta_0\Vert$ for any $\theta\in \theta_0+K/\sqrt{n}$. This yields
\begin{equation*}
  \underset{\theta\in \theta_0+K/\sqrt{n}}{\operatorname{sup}} \dfrac{ \vert\nu_n r(\cdot,\theta)\vert}{1+ \Vert K\Vert}\leq\underset{\theta\in\theta_0+K/\sqrt{n}}{\operatorname{sup}}\hspace{0,1cm} \dfrac{\vert\nu_n r(\cdot,\theta)\vert}{1+n^{1/2}\Vert\theta-\theta_0\Vert}\leq \underset{\theta\in U_n}{\operatorname{sup}}\hspace{0,1cm} \dfrac{\vert\nu_n r(\cdot,\theta)\vert}{1+n^{1/2}\Vert\theta-\theta_0\Vert}\xrightarrow[n\rightarrow\infty]{p} 0,
\end{equation*}
by (\ref{stochastic differentiability}). Hence $\operatorname{sup}_{\theta\in U_n}\vert\nu_n r(\cdot,\theta)\vert$ goes to zero in probability, because $\Vert K\Vert$ is a fixed constant.
\subsection{Extention to pattern preserving regularizers}
Theorem~\ref{main pattern robust theorem} extends to penalties of the form $\tilde{f}(\theta)=f(\rho(\theta))$, where $\rho_{\theta_0}'$ is a differentiable map that ``preserves'' the pattern space around $\rho(\theta_0)$.
\begin{corollary}
    Assume that conditions in Theorem \ref{main pattern robust theorem} hold. Consider a minimizer \eqref{main objective} with a composed penalty $\tilde{f}(\theta)=f(\rho(\theta))$, where $f$ takes the form of the polyhedral gauge in \eqref{penalty form} and  $\rho: \Theta \mapsto \mathbb{R}^m$ is a map differentiable at $\theta_0$. Assume that for any $u_1,u_2\in\mathbb{R}^p$ such that $I_f(u_1)=I_f(u_2)$, we have
    \begin{equation}\label{pattern preservation}
        %\lim_{\varepsilon\downarrow0} 
        I_f(\rho(\theta_0)+\varepsilon \rho_{\theta_0}'(u_1))=I_f(\rho(\theta_0)+\varepsilon \rho_{\theta_0}'(u_2))
    \end{equation}
    for all small $\varepsilon>0$.    
    %then the pattern space of $\rho_{\theta_0}'(u)$ is the same as the pattern space of $u$ for any $u\in supp(\hat{u})$, 
    Then $\hat{u}_n=\sqrt{n}(\hat{\theta}_n-\theta_0)$ converges in distribution to $\hat{u}$, which minimizes
    \begin{equation*}
    V(u) = \dfrac{1}{2}u^{T}Cu-u^{T}W+\tilde{f}'({\theta_0};u),
\end{equation*}
    where $\tilde{f}'(\theta_0;u)= f'(\rho(\theta_0);\rho'_{\theta_0}(u))$, and $I_f(\hat{u}_n)$ converges to $I_f(\hat{u})$, as in \eqref{pattern convergence}.
\end{corollary}
\begin{proof}
    The directional derivative in (\ref{main objective V(u)}) becomes $\tilde{f}'(\theta_0;u)= f'(\rho(\theta_0);\rho'_{\theta_0}(u))$ and the subdifferential $\partial_u \tilde{f}'(\theta_0;u)=(\rho_{\theta_0}')^T \partial f'(\rho(\theta_0); \rho_{\theta_0}'(u))$. Steps 1. to 4. in the proof follow analogously for $\tilde{f}$ as for $f$. Consequently, $\hat{u}_n$ converges in distribution to $\hat{u}$, and we can assume Skorokhod coupling so that $\hat{u}_n$ converges to $\hat{u}$ almost surely. This means that $I(P\hat{u}_n)=I(\hat{u})$, for all sufficiently large $n$, where $P$ is a projection onto the $f-$ pattern space of $\hat{u}$. Now, crucially, for any $u_1,u_2$ with $I(u_1)=I(u_2)$, \eqref{pattern preservation} implies that $\partial f'(\rho(\theta_0);\rho'_{\theta_0}(u_1))=\partial f'(\rho(\theta_0);\rho'_{\theta_0}(u_2))$, for details see the connection between the subdifferential of the directional derivative and the limiting pattern in \cite{hejny2025unveiling}.  Consequently,
    \begin{align*}
        \partial \tilde{f}'(\theta_0;P\hat{u}_n)&=(\rho_{\theta_0}')^T\partial f'(\rho(\theta_0);\rho'_{\theta_0}(P\hat{u}_n))\\&=(\rho_{\theta_0}')^T\partial f'(\rho(\theta_0);\rho'_{\theta_0}(\hat{u}))=\partial \tilde{f}'(\theta_0;\hat{u})
    \end{align*} as in (\ref{key subdifferential equation in Step 5}) and also Step 5. follows. Consequently, Theorem \ref{main pattern robust theorem} holds for the penalty $\tilde{f}$, and $\hat{u}_n$ converges in distribution and in $f-$ pattern to $\hat{u}$.
\end{proof}

\subsection{Optimality lemma for polyhedral gauges}
\begin{lemma}\label{lemma on subdifferential variational property}
Let $v_1,\dots,v_k\in\mathbb R^p$ be fixed vectors and set
\[
h(u)=\max\{v_1^Tu,\dots,v_k^Tu\}.
\]
Let $C\in\mathbb R^{p\times p}$ be symmetric positive definite, and for $w\in\mathbb R^p$ let $\hat u(w)$ be the unique minimizer of
\begin{equation}\label{lemma objective function}
\frac12 u^TCu-u^Tw+h(u).
\end{equation}
Let $W$ be a random vector in $\mathbb R^p$ with bounded density w.r.t. Lebesgue measure. Then for any $\varepsilon>0$ there exists $\delta>0$ such that
\begin{equation*}
\mathbb{P}\big[B_{\delta}^*(0)\subset C\hat{u}(W)-W+\partial h(\hat{u}(W))\big]>1-\varepsilon,
\end{equation*}
where $B_{\delta}^*(0)$ is the Euclidean ball of radius $\delta$ inside the linear subspace $\mathrm{par}(\partial h(\hat{u}(W)))$ parallel to the affine hull of $\partial h(\hat{u}(W))$.
\end{lemma}
\begin{proof}
To prove the result, we split the argument into several steps. First, we partition the range of $W$ into disjoint, relatively open polyhedral cells, indexed by the active set of the minimizer of \eqref{lemma objective function}. Second, since $W$ has a bounded density, with probability arbitrarily close to one, neither $W$ nor $r(W) := W - C\hat{u}(W)$ lies on a cell (or face) boundary. On each cell, the map $w \mapsto \hat{u}(w)$ is affine, and hence $r(w)$ is affine as well. Restricting to a compact set $K_\eta$ at a positive distance from the boundaries and using continuity yields a uniform $\delta > 0$ such that $d(r(w), \partial S_I) \ge \delta$, where $S_I = \operatorname{conv}\{v_i : i \in I\}$. Translating this interior margin gives
\[
B_\delta^*(0)\subset C\hat u(W)-W+\partial h(\hat u(W))
\]
with probability at least $1-\varepsilon$.
\vskip 0.3cm
\textbf{1. Partition into polyhedral cells.} \\
Let 
\begin{equation*}
    \mathcal{R}_I:=\{w: I(\hat{u}(w))=I\},
\end{equation*} 
where $I(u):=\{i: v_i^Tu=h(u)\}$ is the $h-$pattern at $u\in\mathbb{R}^p$. We show the sets $\{\mathcal R_I\}_I$ form a finite partition of $\mathbb R^p$ into relatively open polyhedral cells.
    
Fix a pattern $I$. Let $V_I$ be the $p\times |I|$ matrix whose columns are the $v_i$ for $i\in I$, and let $V_{I^c}$ be the $p\times (k-|I|)$ matrix with columns $v_i$ for $i\notin I$. By first-order (KKT) optimality, $u$ is a minimizer of \eqref{lemma objective function} if and only if
\begin{equation*}
    w-Cu\in\partial h(u)= \operatorname{con}\{v_i: i\in I(u)\}.
\end{equation*}

Suppose $u$ is a minimizer of \eqref{lemma objective function} with pattern $I(u)=I$. Set $t:=h(u)$. By definition of the pattern we have $
V_I^T u = t\mathbf1$ and $V_{I^c}^T u < t\mathbf1$. Moreover the KKT condition above gives $w-Cu\in\operatorname{conv}\{v_i:i\in I\}$, hence there exists $\mu\ge0$ with $\mathbf1^T\mu=1$ such that $w-Cu=V_I\mu$. Thus the tuple $(u,t,\mu,w)$ satisfies the linear system

\begin{equation}\label{optimality polyhedron}
\begin{cases}
V_I^Tu = t\mathbf{1},\\[2pt]
V_{I^c}^Tu < t\mathbf{1},\\[2pt]
w-C u = V_I\mu,\\[2pt]
\mu\ge 0,\ \mathbf{1}^T\mu = 1.
\end{cases}
\end{equation}
Conversely, given any fixed $w\in\mathbb{R}^p$, suppose there exist $u,t,\mu$ such that $(u,t,\mu,w)$ satisfies \eqref{optimality polyhedron}. Then $V_I^T u=t\mathbf1$ and $V_{I^c}^Tu<t\mathbf1$ imply $h(u)=t$ and $I(u)=I$. Also $w-Cu=V_I\mu\in\operatorname{conv}\{v_i:i\in I\}$, so $w-Cu\in\partial h(u)$, hence $u$ is a minimizer of \eqref{lemma objective function} with pattern $I$, and $w\in\mathcal{R}_I$. 

Let $\mathcal P_I$ be the set\footnote{Note that for patterns $I$ for which there is not $u$ with $I(u)=I$, the set $\mathcal{P}_{I}$ is empty. } of all $(u,t,\mu,w)$ satisfying \eqref{optimality polyhedron}. Then
\begin{equation*}
    \mathcal{R}_I=\{w: \exists\,u,t,\mu \text{ such that } (u,t,\mu,w)\in \mathcal{P}_I\}.
\end{equation*}
(Here $\mathcal P_I$ is relatively open because of the strict inequalities; replacing $<$ by $\le$ yields the closed polyhedron $\overline{\mathcal P_I}$.) The projection $\pi(\overline{\mathcal P_I})$ onto the $w$-coordinates equals $\overline{\mathcal R_I}$, hence each $\overline{\mathcal R_I}$ is a polyhedron. Since there are finitely many patterns $I$, the sets $\{\mathcal R_I\}_I$ give a finite partition of $\mathbb R^p$ into relatively open cells whose closures are polyhedra.
\vskip 0.4cm

\textbf{2. $W$ avoids the boundaries of $\mathcal R_I$.} \\
Let $\mathcal B=\bigcup_I\partial\mathcal R_I$ denote the union of all cell boundaries. Each $\partial\mathcal R_I$ is contained in a finite union of lower-dimensional polyhedral sets, so $\mathcal B$ has Lebesgue measure zero. Since $W$ has a bounded density w.r.t. Lebesgue measure, for any $\eta>0$ the tubular neighborhood $\mathcal B^\eta=\{w: d(w,\mathcal{B})\leq \eta\}$ satisfies $\mathbb P[W\in\mathcal B^\eta]\to 0$ as $\eta\downarrow0$. (Here we use the standard Euclidean metric $d(x,y)=\|x-y\|$.)

Fix $\varepsilon>0$ and choose $R>0$ such that $\mathbb P(\|W\|\le R)>1-\varepsilon/2$. Define
\begin{equation}\label{compact set away from boundary}
K_{\eta}=\{w\in\mathbb R^p:\ \|w\|\le R,\ d(w,\mathcal B)\ge\eta\},
\end{equation}
and pick $\eta>0$ sufficiently small so that $\mathbb P[W\in K_{\eta}]\ge 1-\varepsilon$.
\vskip 0.5cm

\vskip 0.4cm
\textbf{3. Interior points.} \\
For a given index set $I$ denote $S_I:=\operatorname{conv}\{v_i:i\in I\}=\partial h(u)$ when $I=I(u)$. For $w\in\mathcal R_I$ write $\hat u=\hat u(w)$ and define
\[
r(w):=w-C\hat u(w)\in S_I.
\]
Let $\mathrm{ri}(A)$ denotes the \emph{relative interior} of a set $A\subset\mathbb R^p$,
we show
\begin{equation}\label{relative interior implication}
w\in \mathrm{ri}(\mathcal R_I)\quad\Longrightarrow\quad r(w)\in \mathrm{ri}(S_I).
\end{equation}
Suppose for contradiction that $w\in\mathrm{ri}(\mathcal R_I)$ but $r(w)\notin\mathrm{ri}(S_I)$. Then $r(w)$ lies in the relative boundary of $S_I$, hence there exists a strict subset $J\subsetneq I$ and weights $\rho\ge0$, $\sum_{j\in J}\rho_j=1$, supported on $J$, such that $r(w)=V_J\rho$. But then $(\hat u,t,\rho,w)$ satisfies the non-strict system \eqref{optimality polyhedron} for index set $J$, so $w\in\overline{\mathcal R_J}$. Since $\mathcal R_I\cap\mathcal R_J=\varnothing$ for $I\ne J$, this implies $w\in\partial\mathcal R_I$, contradicting $w\in\mathrm{ri}(\mathcal R_I)$. Thus \eqref{relative interior implication} holds.

\vskip 0.4cm
\textbf{4. Compactness argument and uniform margin.} \\
On the interior of each cell $\mathcal R_I$ the mapping $w\mapsto \hat u(w)$ is affine (hence continuous), so the map $r(w)=w-C\hat u(w)$ is affine and the function
\[
g_I(w):=d\big(r(w),\partial S_I\big)
\]
is continuous on $\operatorname{ri}(\mathcal R_I)$ (distance measured within the affine hull of $S_I$). Define $g(w)=g_I(w)$ for $w\in\operatorname{ri}(\mathcal R_I)$.

On the compact set $K_\eta$ defined in \eqref{compact set away from boundary} every point lies in the relative interior of some cell, and by \eqref{relative interior implication} we have $g(w)>0$ for every $w\in K_\eta$. By continuity and compactness there exists $\delta>0$ such that
\[
g(w)\ge \delta\quad\text{for all }w\in K_\eta,
\]
i.e.
\[
d\big(w-C\hat u(w),\partial S_I\big)\ge\delta\quad\text{whenever }w\in K_\eta\cap\mathcal R_I.
\]
Equivalently, for such $w$,
\[
\big(w-C\hat u(w)\big)+B_\delta^*(0)\subset S_I,
\]
where $B_\delta^*(0)$ denotes the Euclidean ball of radius $\delta$ inside the linear subspace $\mathrm{par}(S_I)=\mathrm{par}(\partial h(\hat u(w)))$. Translating,
\[
B_\delta^*(0)\subset C\hat u(w)-w + \partial h(\hat u(w)).
\]

\vskip 0.4cm
\textbf{5. Conclude.} \\
Since $\mathbb P[W\in K_\eta]\ge 1-\varepsilon$, for the $\delta>0$ above we obtain
\[
\mathbb{P}\big[B_{\delta}^*(0)\subset C\hat{u}(W)-W+\partial h(\hat{u}(W))\big]
\ \ge\ \mathbb P[W\in K_\eta]\ \ge\ 1-\varepsilon,
\]
which completes the proof.

\end{proof}

\end{appendix}
\end{document}